\documentclass{amsart}
\usepackage{amsmath}
\usepackage{amssymb}
\usepackage{amsthm}
\usepackage{epsfig}
\usepackage{amsfonts}
\usepackage{amscd}
\usepackage{mathrsfs}
\usepackage{enumerate}
\usepackage{latexsym}
\usepackage{url}

\newtheorem{theorem}{Theorem}[section]
\newtheorem{lemma}[theorem]{Lemma}

\newtheorem{conjecture}[theorem]{Conjecture}

\theoremstyle{definition}
\newtheorem{definition}[theorem]{Definition}
\newtheorem{example}[theorem]{Example}

\theoremstyle{remark}
\newtheorem{remark}[theorem]{Remark}

\theoremstyle{notation}
\newtheorem{notation}[theorem]{Notation}
\numberwithin{equation}{section}

\begin{document}

\title[The partially ordered set of one-point extensions]{The partially ordered set of one-point extensions}

\author{M.R. Koushesh}
\address{Department of Mathematical Sciences, Isfahan University of Technology, Isfahan 84156--83111, Iran}
\address{School of Mathematics, Institute for Research in Fundamental Sciences (IPM), P.O. Box: 19395--5746, Tehran, Iran}
\email{koushesh@cc.iut.ac.ir}
\thanks{This research was in part supported by a grant from IPM (No. 86540012).}



\begin{abstract}
A space $Y$ is called an {\em extension} of a space $X$ if $Y$
contains $X$ as a dense subspace.
Two extensions of $X$ are said to be {\em equivalent} if there is a homeomorphism between them which fixes $X$ point-wise.
For two (equivalence classes of)  extensions $Y$ and $Y'$ of $X$ let $Y\leq Y'$ if there is a continuous function of $Y'$ into $Y$
which fixes  $X$ point-wise.
An extension $Y$ of  $X$ is called a {\em one-point extension} of $X$ if $Y\backslash X$ is a singleton. Let ${\mathcal P}$ be a topological
property. An extension $Y$ of $X$ is called a {\em
${\mathcal P}$-extension} of $X$ if it has ${\mathcal P}$.

One-point ${\mathcal P}$-extensions comprise the subject matter of this article. Here ${\mathcal P}$
is subject to some mild requirements. We define an anti-order-isomorphism between the set of one-point Tychonoff extensions of a (Tychonoff) space $X$ (partially ordered by $\leq$) and the set of compact non-empty subsets of its outgrowth $\beta X\backslash X$ (partially ordered by $\subseteq$). This enables us to study the order-structure of  various sets of one-point extensions of the space $X$ by relating them to the  topologies of certain subspaces of its outgrowth. We conclude the article with the following conjecture. For a Tychonoff spaces $X$ denote by ${\mathscr U}(X)$ the set of all zero-sets of $\beta X$ which miss $X$.\\

\noindent{\bf Conjecture.} {\em For locally compact spaces $X$ and $Y$ the partially ordered sets $({\mathscr U}(X),\subseteq)$ and $({\mathscr U}(Y),\subseteq)$ are order-isomorphic if and only if the spaces $\mbox{\em cl}_{\beta X}(\beta X\backslash\upsilon X)$ and $\mbox{\em cl}_{\beta Y}(\beta Y\backslash\upsilon Y)$ are homeomorphic.}
\end{abstract}

\maketitle

\section{Introduction}

A space $Y$ is called an {\em extension} of a space $X$ if $Y$
contains $X$ as a dense subspace.  If $Y$ is an extension of  $X$ then the subspace $Y\backslash X$ of $Y$ is called the
{\em  remainder} of $Y$.  Extensions with a one-point remainder are called {\em one-point extensions}. Two extensions of a space $X$ are said to be {\em equivalent} if there is a homeomorphism between them
which fixes $X$ point-wise. This defines an equivalence relation on the class of all extensions of a space $X$. The equivalence classes
will be identified with individuals when this causes no confusion. For two (equivalence classes of)  extensions $Y$ and $Y'$ of $X$, we
let $Y\leq Y'$, if there is a continuous function of $Y'$ into $Y$ which fixes  $X$ point-wise. The relation $\leq$ defines  a partial
order on the set of (equivalence classes of) extensions of $X$ (see Section 4.1 of \cite{PW} for more details). Let ${\mathcal P}$ be a topological
property. An extension $Y$ of $X$ is called a {\em
${\mathcal P}$-extension} of $X$, if it possesses ${\mathcal P}$. If ${\mathcal P}$ is compactness, then
${\mathcal P}$-extensions are called  {\em compactifications}. One-point ${\mathcal P}$-extensions are the subject matter of this article.

This work was mainly  motivated by our previous work \cite{Ko2} (see also \cite{Be}, \cite{HJW}, \cite{Ko1}, \cite{Ko4}, \cite{MRW1} and \cite{MRW2}) in which we have studied the partially ordered set of one-point ${\mathcal P}$-extensions of a given locally compact space $X$, by relating it to the topologies of subspaces of its outgrowth $\beta X\backslash X$. Topological properties ${\mathcal P}$  considered in \cite{Ko2}, however, were limited  (mainly to the Lindel\"{o}f property). Here, besides some new results, we generalize most of the results of \cite{Ko2} by studying  the partially ordered set of one-point ${\mathcal P}$-extensions of a space $X$ (where ${\mathcal P}$ is subject to some mild requirements, and ranges over a reasonably  broad class of topological properties as special cases).

This article is organized as follows. In Section 2, we define an anti-order-isomorphism $\Theta_X$ between the set of  one-point Tychonoff extensions of a (Tychonoff) space $X$ (partially ordered by $\leq$) and the set of compact non-empty subsets of its outgrowth $\beta X\backslash X$ (partially ordered by $\subseteq$). We then obtain images (under  $\Theta_X$) of various sets of one-point  extensions of $X$. In Section 3, we prove our first main  result: The set of one-point Tychonoff ${\mathcal P}$-extensions of a space $X$ contains an anti-order-isomorphic copy  of the set of its one-point first-countable Tychonoff extensions. In  Section 4, we study extensions and restrictions of order-isomorphisms between various sets of one-point ${\mathcal P}$-extensions of spaces $X$ and $Y$. In  Section 5, we study the relations between the order-structure of sets of one-point  ${\mathcal P}$-extensions of a space  $X$ and the topologies of subspaces of its outgrowth $\beta X\backslash X$. In Section 6, we consider the set of one-point  ${\mathcal P}$-extensions of a space $X$ as a lattice. And finally, in  Section 7, we  conclude  with  a conjecture which naturally arises in connection with our studies.

We now review some terminology, notation, and well known results that we will  use in the sequel. Our definitions mainly  come from the standard  text \cite{E} (thus, in particular, compact spaces are Hausdorff,  etc.). Other useful sources are \cite{GJ}, \cite{PW} and \cite{W}.

The letters $\mathbf{I}$,  $\mathbf{N}$ and $\mathbf{Q}$ denote the closed unit interval,  the set of all positive integers and the set of all rational numbers, respectively.

For a subset $A$ of a space $X$, we let $\mbox{cl}_X A$ and $\mbox{int}_X A$ to denote the closure and the interior  of the set $A$ in $X$, respectively.

A subset of a space is called {\em clopen}, if it is simultaneously closed and open.
A {\em zero-set} of a space $X$ is a set of the form
$Z(f)=f^{-1}(0)$, for some continuous $f:X\rightarrow \mathbf{I}$.  Any set of the form
$X\backslash Z$, where $Z$ is a zero-set of $X$, is called a {\em
cozero-set} of $X$. We denote the set of all zero-sets of $X$ by
${\mathscr Z}(X)$, and the set of all cozero-sets of $X$ by $Coz(X)$.

For a Tychonoff space $X$, the
{\em Stone-\v{C}ech compactification} of $X$ is the largest (with respect to the partial order of $\leq$)
compactification of $X$, and is  denoted by $\beta X$.  The Stone-\v{C}ech compactification of a Tychonoff $X$ is characterized
among the compactifications of $X$ by either of the following properties:
\begin{itemize}
\item[\rm(1)] Every continuous function of $X$ to a compact space (or $\mathbf{I}$) is continuously extendible over $\beta X$.
\item[\rm(2)] For every $Z,S\in {\mathscr Z}(X)$ we have
\[\mbox{cl}_{\beta X}(Z\cap S)=\mbox{cl}_{\beta X}Z\cap\mbox{cl}_{\beta X}S.\]
\end{itemize}

For a Tychonoff space $X$, the {\em Hewitt realcompactification} of $X$ (denoted by $\upsilon X$) is the intersection of all cozero-sets of $\beta X$ which contain $X$.

A Tychonoff space $X$ is called {\em  \v{C}ech-complete}, if $X$ is a $G_\delta$ in $\beta X$.  Locally compact spaces are  \v{C}ech-complete, and in the realm of metrizable spaces $X$, \v{C}ech-completeness is equivalent to the existence of a compatible complete metric on $X$.

A Tychonoff space $X$ is called {\em  zero-dimensional}, if it has an open base consisting of clopen subsets of $X$. A Tychonoff space $X$ is called {\em strongly zero-dimensional}, if its Stone-\v{C}ech compactification $\beta X$ is zero-dimensional.

Let ${\mathcal P}$ be a topological property. A  space $X$ is called {\em locally-${\mathcal P}$}, if for every $x\in X$
there is an open neighborhood $U_x$ of $x$ in $X$ such that $\mbox{cl}_XU_x$ has ${\mathcal P}$.

A topological property $\mathcal{P}$  is said to be {\em hereditary with respect to closed subsets} if each closed subset of a space with $\mathcal{P}$ also has $\mathcal{P}$. A  topological property ${\mathcal P}$ is said to be  {\em preserved under finite} ({\em locally finite}, {\em countable}, respectively)  {\em closed sums of subspaces}, if any
Hausdorff space which is expressible as a  finite (locally finite, countable, respectively) union of its closed
${\mathcal P}$-subsets has ${\mathcal P}$.

In a partially ordered set $(P,\leq)$ the two symbols $\vee$ and $\wedge$ denote the least upper bound and the greatest lower bound (provided they exist), respectively. A partially ordered set $(P,\leq)$ is called a {\em lattice} if together with each pair of elements $a,b\in P$ it contains $a\vee b$ and $a\wedge b$. A partially ordered set $(P,\leq)$ is called a {\em complete upper semilattice} ({\em complete lower semilattice}, respectively)  if for each non-empty $A\subseteq P$ the least upper bound $\bigvee A$ (the greatest lower bound $\bigwedge A$, respectively) exists in $P$. A partially ordered set $(P,\leq)$ is called a {\em complete lattice} if it is both a complete upper semilattice and a complete lower semilattice.

For partially ordered sets $(P,\leq)$ and $(Q,\leq)$, a function $f:(P,\leq)\rightarrow(Q,\leq)$ is called an {\em order-homomorphism} ({\em anti-order-homomorphism}, respectively), if $f(a)\leq f(b)$ ($f(b)\leq f(a)$, respectively) whenever $a\leq b$. An order-homomorphism (anti-order-homomorphism, respectively) $f:(P,\leq)\rightarrow(Q,\leq)$ is called an {\em order-isomorphism} ({\em anti-order-isomorphism}, respectively), if $f^{-1}:(Q,\leq)\rightarrow(P,\leq)$ (exists and) is an order-homomorphism (anti-order-homomorphism, respectively). Two partially ordered sets $(P,\leq)$ and $(Q,\leq)$ are said to be {\em order-isomorphic} ({\em anti-order-isomorphic}, respectively), if there is an  order-isomorphism (anti-order-isomorphism, respectively) between them.

\section{Basic definitions and preliminary results}

In what follows, we will be dealing with various sets of one-point extensions of a given space $X$.
For the reader's convenience, we list these sets all at the beginning. First, a definition.

\begin{definition}\label{JSB}
{\em Let $X$ be a space and let ${\mathcal P}$ be a topological property.
An extension $Y$ of $X$ is called a {\em ${\mathcal P}$-far extension} of $X$, if for every  $Z\in {\mathscr Z}(X)$ we have
$\mbox{cl}_Y Z\cap (Y\backslash X)=\emptyset$, whenever
$Z\subseteq C$  for some $C\in Coz(X)$ such that $\mbox{cl}_X C$ has ${\mathcal P}$.}
\end{definition}

\begin{notation}\label{JHB}
{\em For a space $X$ and  a topological property  ${\mathcal P}$, denote
\begin{itemize}
  \item ${\mathscr E}(X)=\{Y:Y \mbox{ is a one-point Tychonoff extension of } X\}$
  \item ${\mathscr E}^*(X)=\{Y\in{\mathscr E}(X):Y \mbox{ is first-countable at } Y\backslash X\}$
  \item ${\mathscr E}^K(X)=\{Y\in{\mathscr E}(X):Y \mbox{ is locally compact}\}$
  \item ${\mathscr E}^C(X)=\{Y\in{\mathscr E}(X):Y \mbox{ is \v{C}ech-complete}\}$
  \item ${\mathscr E}_{{\mathcal P}}(X)=\{Y\in{\mathscr E}(X) :Y \mbox{ has  } {\mathcal P}\}$
  \item ${\mathscr E}_{{\mathcal P}-far}(X)=\{Y\in{\mathscr E}(X) :Y \mbox{ is ${\mathcal P}$-far}\}$.
\end{itemize}
Also, we may use notations obtained by combining the above notations, e.g.
\[{\mathscr E}^C_{{\mathcal P}}(X)={\mathscr E}^C(X)\cap{\mathscr E}_{{\mathcal P}}(X).\]}
\end{notation}

\begin{notation}\label{KJB}
{\em For a Tychonoff space $X$  and  $Y\in {\mathscr E}(X)$ denote by
\[\tau_Y:\beta X\rightarrow \beta Y\]
the unique continuous extension of $\mbox{id}_X$.}
\end{notation}

Theorem \ref{JHV} establishes  a connection between one-point Tychonoff extensions of a (Tychonoff) space $X$ and compact non-empty
subsets of its outgrowth $\beta X\backslash X$. Theorem \ref{JHV} (and its preceding lemmas) is known (see e.g. \cite{MRW1}), the  proof  is given  however, for the sake of completeness.

\begin{lemma}\label{KFH}
Let $X$ be a Tychonoff  space and let $C$ be a non-empty compact subset of $\beta X\backslash X$. Let $T$ be the space which is obtained from $\beta X$
by contracting $C$ to a point $p$. Then the subspace $Y=X\cup\{p\}$ of $T$ is Tychonoff and $\beta Y=T$.
\end{lemma}

\begin{proof}
Let $q:\beta X\rightarrow T$ be the quotient mapping. Note that $T$ is Hausdorff, and therefore, being the continuous image of $\beta X$, it is compact.
Also, note that $Y$ is dense in $T$. Thus $T$ is a compactification of $Y$. To show that $\beta Y=T$, it suffices to verify that every continuous
$h:Y\rightarrow{\mathbf I}$ is continuously extendable over $T$. Let $h:Y\rightarrow{\mathbf I}$ be  continuous.  Let $G:\beta X\rightarrow{\mathbf I}$  continuously extend $hq|(X\cup C):X\cup C\rightarrow{\mathbf I}$
(note that $\beta(X\cup C)=\beta X$, as $X\subseteq X\cup C\subseteq \beta X$; see Corollary 3.6.9 of \cite{E}). Define  $H:T\rightarrow{\mathbf I}$ such that
$H|(\beta X\backslash C)=G|(\beta X\backslash C)$ and $H(p)=h(p)$. Then $H|Y=h$, and since $Hq=G$ is continuous, the function $H$ is continuous.
\end{proof}

\begin{lemma}\label{JFV}
Let $X$ be a Tychonoff space and let $Y=X\cup\{p\}\in {\mathscr E}(X)$. Let $T$ be the
space which is obtained from $\beta X$ by contracting  $\tau_Y^{-1}(p)$ to a
point $p$, and let $q:\beta X\rightarrow T$ be the
quotient mapping. Then $T=\beta Y$  and $\tau_Y=q$.
\end{lemma}

\begin{proof}
We need to show that $Y$ is a subspace of $T$. Since $\beta Y$ is also a compactification of $X$ and
$\tau_Y|X=\mbox{id}_X$, by Theorem 3.5.7 of \cite{E} we have  $\tau_Y(\beta X\backslash X)=\beta Y\backslash X$.  For an open subset $W$ of $\beta Y$, the set  $q(\tau_Y^{-1}(W))$ is open in $T$, as $q^{-1}(q(\tau_Y^{-1}(W)))=\tau_Y^{-1}(W)$ is open in $\beta X$. Thus $Y\cap W= Y\cap q(\tau_Y^{-1}(W))$ is open in $Y$, when $Y$ is considered as  a subspace of $T$. For the converse, note that if $V$ is open in $T$, since
\[Y\cap V=Y\cap \big(\beta Y\backslash\tau_Y\big(\beta X\backslash q^{-1}(V)\big)\big)\]
and  $\tau_Y(\beta X\backslash q^{-1}(V))$ is compact and therefore closed in $\beta Y$, the set $Y\cap V$ is open in $Y$ in its original topology.
By Lemma \ref{KFH} we have $T=\beta Y$. This
also implies that $\tau_Y=q$, as $\tau_Y,q:\beta X\rightarrow\beta Y$ are
continuous and coincide with $\mbox{id}_X$ on the dense subset $X$ of $\beta X$.
\end{proof}

\begin{lemma}\label{DFH}
Let $X$ be a Tychonoff space. Let $Y_i\in{\mathscr E}(X)$, for $i=1,2$, and denote by $\tau_i=\tau_{Y_i}:\beta X\rightarrow \beta Y_i$ the continuous
extension of $\mbox{\em id}_X$. The following are equivalent:
\begin{itemize}
\item[\rm(1)] $Y_1\leq Y_2$.
\item[\rm(2)] $\tau^{-1}_2(Y_2\backslash X)\subseteq\tau^{-1}_1(Y_1\backslash X)$.
\end{itemize}
\end{lemma}

\begin{proof}
Let $Y_i=X\cup\{p_i\}$ for $i=1,2$.  (1) {\em  implies} (2).  By definition, there is a continuous
$f:Y_2\rightarrow Y_1$ such that $f|X= \mbox{id}_X$. Let $f_\beta:\beta Y_2\rightarrow \beta Y_1$ continuously  extend $f$. Note that the continuous
functions $f_\beta\tau_2,\tau_1:\beta X\rightarrow \beta Y_1$ coincide  with $\mbox{id}_X$ on the dense subset $X$ of $\beta X$, and thus
$f_\beta\tau_2= \tau_1$. Also, note that $X$ is dense in $\beta Y_i$ (for $i=1,2$), as it is  dense in $Y_i$, and therefore,  $\beta Y_i$ is a compactification
of $X$. Since $f_\beta|X=\mbox{id}_X$, by Theorem 3.5.7 of \cite{E} we have $f_\beta(\beta Y_2\backslash X)=\beta Y_1\backslash X$, and thus $f_\beta(p_2)\in \beta Y_1\backslash X$. But $f_\beta(p_2)=f(p_2)$, which implies that $f_\beta(p_2)\in Y_1\backslash X=\{p_1\}$.  Therefore
\[\tau^{-1}_2(p_2)\subseteq\tau^{-1}_2\big(f_\beta^{-1}\big(f_\beta(p_2)\big)\big)=(f_\beta\tau_2)^{-1}\big(f_\beta(p_2)\big)
=\tau_1^{-1}\big(f_\beta(p_2)\big)=\tau_1^{-1}(p_1).\]

(2) {\em  implies} (1).  Let $f:Y_2\rightarrow Y_1$ be defined  such that $f(p_2)=p_1$ and $f|X=\mbox{id}_X$. We show that $f$ is continuous, this will show that $Y_1\leq Y_2$. Note that by Lemma \ref{JFV}, the space $\beta Y_2$ is the quotient space of $\beta X$ which is  obtained by contracting  $\tau^{-1}_2(p_2)$ to a point, and $\tau_2$ is its corresponding quotient mapping. Thus, in particular, $Y_2$ is the quotient space of $X\cup \tau^{-1}_2 (p_2)$, and therefore, to show that $f$ is continuous, it suffices to show that $f\tau_2|(X\cup\tau^{-1}_2 (p_2))$ is continuous. We show this by verifying that $f\tau_2(t)=\tau_1(t)$ for  $t\in X\cup\tau^{-1}_2 (p_2)$. This obviously holds if $t\in X$. If $t\in \tau^{-1}_2 (p_2)$, then $\tau_2(t)=p_2$, and thus $f\tau_2 (t)=p_1$. But since $t\in \tau^{-1}_2(\tau_2 (t))$, we have $t\in \tau^{-1}_1(p_1)$, and therefore $\tau_1(t)=p_1$. Thus $f\tau_2(t)=\tau_1(t)$ in this case as well.
\end{proof}

\begin{theorem}\label{JHV}
Let $X$ be a Tychonoff space. Define a function
\[\Theta:\big({\mathscr E}(X),\leq\big)\rightarrow\big(\{C\subseteq\beta X\backslash X:C \mbox{ is compact} \}\backslash\{\emptyset\},\subseteq\big)\]
by
\[\Theta(Y)=\tau^{-1}_Y(Y\backslash X)\]
for $Y\in{\mathscr E}(X)$. Then $\Theta$ is an anti-order-isomorphism.
\end{theorem}

\begin{proof}
To show that  $\Theta$ is well-defined, let $Y\in{\mathscr E}(X)$. Note that since $X$ is  dense in $Y$, the space $X$ is dense in  $\beta Y$. Therefore     $\tau_Y:\beta X\rightarrow \beta Y$ is onto, as $\tau_Y(\beta X)$ is a compact subset of $\beta Y$ and it contains $X=\tau_Y(X)$. Thus  $\tau^{-1}_Y(Y\backslash X)\neq\emptyset$. Also,  since $\tau_Y|X=\mbox{id}_X$ we have $\tau^{-1}_Y(Y\backslash X)\subseteq\beta X\backslash X$, and since the singleton $Y\backslash X$ is closed in $\beta Y$, its inverse image $\tau^{-1}_Y(Y\backslash X)$ is closed in $\beta X$, and therefore is compact. Now we show that $\Theta$ is onto,  Lemma \ref{DFH} will then complete the proof. Let $C$  be a non-empty  compact subset of $\beta X\backslash X$. Let $T$ be the quotient space of $\beta X$ which is obtained by contracting $C$ to a point $p$. Consider the subspace $Y=X\cup\{p\}$  of $T$. Then $Y\in{\mathscr E}(X)$, and thus, by Lemma \ref{KFH} we have $\beta Y=T$. The quotient mapping $q:\beta X\rightarrow T$ is identical to $\tau_Y$, as it coincides with $\mbox{id}_X$ on the dense subset $X$ of $\beta X$. Therefore
\[\Theta(Y)=\tau^{-1}_Y(p)=q^{-1}(p)=C.\]
\end{proof}

\begin{notation}\label{GFD}
{\em For a  Tychonoff space $X$ denote by
\[\Theta_X:\big({\mathscr E}(X),\leq\big)\rightarrow\big(\{C\subseteq \beta X\backslash X:C \mbox{ is compact} \}\backslash\{\emptyset\},\subseteq\big)\]
the anti-order-isomorphism defined by
\[\Theta_X(Y)=\tau^{-1}_Y(Y\backslash X)\]
for $Y\in {\mathscr E}(X)$.}
\end{notation}

Lemmas \ref{JYYH}, \ref{FSYH} and \ref{HDJYYH} are from  \cite{Ko2}. We include the proofs here for the sake of completeness.

\begin{lemma}\label{JYYH}
Let $X$ be a Tychonoff  space. For $Y\in{\mathscr E}(X)$ the following are equivalent:
\begin{itemize}
\item[\rm(1)] $Y\in{\mathscr E}^*(X)$.
\item[\rm(2)] $\Theta_X(Y)\in {\mathscr Z}(\beta X)$.
\end{itemize}
\end{lemma}

\begin{proof}
Let $Y=X\cup\{p\}$. (1) {\em  implies} (2). Let $\{V_n:n\in\mathbf{N}\}$ be an open base at $p$ in $Y$. For each
$n\in \mathbf{N}$, let $V'_n$ be an open subset of $\beta Y$ such that $Y\cap V'_n=V_n$, and let $f_n:\beta Y\rightarrow\mathbf{I}$ be continuous and such that $f_n(p)=0$ and $f_n|(\beta Y\backslash V'_n)\equiv 1$. Let
\[Z=\bigcap_{n=1}^\infty Z(f_n)\in {\mathscr Z}(\beta Y).\]
We show that $Z=\{p\}$. Obviously, $p\in Z$. Let $t\in Z$ and suppose to the contrary that $t\neq p$. Let $W$ be an open neighborhood of $p$ in $\beta Y$ such that $t\notin \mbox{cl}_{\beta Y}W$. Then $Y\cap W$ is an open neighborhood of $p$ in $Y$. Let  $k\in \mathbf{N}$ be such that $V_k\subseteq Y\cap W$. We have
\[t\in Z(f_k)\subseteq V'_k \subseteq\mbox{cl}_{\beta Y}V'_k=\mbox{cl}_{\beta Y}(Y\cap V'_k)=\mbox{cl}_{\beta Y}V_k\subseteq\mbox{cl}_{\beta Y}(Y\cap W)\subseteq\mbox{cl}_{\beta Y}W\]
which is a contradiction. This shows that  $t=p$ and therefore $Z\subseteq \{p\}$. Thus $\{p\}=Z\in {\mathscr Z}(\beta Y)$, which implies that
$\tau^{-1}_Y(p)\in {\mathscr Z}(\beta X)$.

(2) {\em  implies} (1). Let $\tau^{-1}_Y(p)=Z(f)$ where $f:\beta X\rightarrow\mathbf{I}$ is continuous.  Note that by Lemma \ref{JFV} the space $\beta Y$ is obtained from $\beta X$ by contracting $\tau^{-1}_Y(p)$ to $p$ with $\tau_Y:\beta X\rightarrow \beta Y$ as the quotient mapping. Then for each $n\in \mathbf{N}$ the set $\tau_Y(f^{-1}([0,1/n)))$ is an open neighborhood of $p$ in $\beta Y$. We show that the collection
\[\big\{Y\cap\tau_Y\big(f^{-1}\big([0,1/n)\big)\big):n\in \mathbf{N}\big\}\]
of open neighborhoods of $p$ in $Y$ constitute an open base at $p$ in $Y$.  Let $V$ be an open neighborhood of $p$ in $Y$.
Let $V'$ be an open subset of $\beta Y$ such that $Y\cap V'=V$. Then $p\in V'$ and thus
\[\bigcap_{n=1}^\infty f^{-1}\big([0,1/n]\big)=Z(f)=\tau_Y^{-1}(p)\subseteq\tau_Y^{-1}(V').\]
By compactness we have $f^{-1}([0,1/k])\subseteq\tau_Y^{-1}(V')$ for some  $k\in \mathbf{N}$. Therefore
\[Y\cap\tau_Y\big(f^{-1}\big([0,1/k)\big)\big)\subseteq Y\cap\tau_Y\big(f^{-1}\big([0,1/k]\big)\big)\subseteq Y\cap\tau_Y\big(\tau_Y^{-1}(V')\big)\subseteq Y\cap V'=V.\]
\end{proof}

\begin{lemma}\label{FSYH}
Let $X$ be a locally compact  space. For $Y\in{\mathscr E}(X)$ the following are equivalent:
\begin{itemize}
\item[\rm(1)] $Y\in{\mathscr E}^K(X)$.
\item[\rm(2)] $\Theta_X(Y)$ is open in $\beta X\backslash X$.
\end{itemize}
\end{lemma}

\begin{proof}
Let $Y=X\cup\{p\}$. (1) {\em  implies} (2). Since $Y$ is locally compact, $Y$ is  open in $\beta Y$. Note that by Lemma \ref{JFV} the space $\beta Y$ is the quotient space of $\beta X$ which is obtained by contracting $\tau^{-1}_Y(p)$ to $p$ and $\tau_Y$ is its quotient mapping. Then the set $\tau^{-1}_Y(p)=\tau^{-1}_Y(Y)\cap (\beta X\backslash X)$ is open in $\beta X\backslash X$.

(2) {\em  implies} (1).  Let $\tau^{-1}_Y(p)=U\cap (\beta X\backslash X)$, for some open subset $U$ of $\beta X$. Note that since $X$ is locally compact, $X$ is open in $\beta X$. Then, again using Lemma \ref{JFV}, it follows that $Y=\tau_Y(U\cup X)$ is open in $\beta Y$, from which the local compactness of $Y$ follows.
\end{proof}

\begin{lemma}\label{HDJYYH}
Let $X$ be a locally compact  space. For $Y\in{\mathscr E}(X)$ the following are equivalent:
\begin{itemize}
\item[\rm(1)] $Y\in{\mathscr E}^C(X)$.
\item[\rm(2)] $\Theta_X(Y)\in {\mathscr Z}(\beta X\backslash X)$.
\end{itemize}
\end{lemma}

\begin{proof}
Let $Y=X\cup\{p\}$. (1) {\em  implies} (2).  The set $\beta Y\backslash Y$ is an
$F_\sigma$ in $\beta Y$. Let $\beta Y\backslash Y=\bigcup_{n=1}^\infty K_n$ where each $K_n$ for $n\in\mathbf{N}$ is closed in $\beta Y$. Note that by Lemma \ref{JFV} the space $\beta Y$ is the quotient space of $\beta X$ which is obtained by contracting $\tau^{-1}_Y(p)$ to $p$ and $\tau_Y$ is its quotient mapping. Then
\[\beta X\backslash X=\tau^{-1}_Y(p)\cup\bigcup_{n=1}^\infty K_n.\]
For each  $n\in \mathbf{N}$, let  $f_n:\beta X\rightarrow \mathbf{I}$ be continuous and  such that $f_n|\tau^{-1}_Y(p)\equiv 0$ and $f_n|K_n\equiv 1$. Let $f=\sum_{n=1}^\infty f_n/2^n$. Then $f:\beta X\rightarrow \mathbf{I}$ is continuous and
\[\tau^{-1}_Y(p)=Z(f)\cap(\beta X\backslash X)\in{\mathscr Z}(\beta X\backslash X).\]

(2) {\em  implies} (1).  Let  $\tau^{-1}_Y(p)=Z(g)$ where  $g:\beta X\backslash X\rightarrow \mathbf{I}$ is continuous. Then, again using  Lemma \ref{JFV}, we have
\[\beta Y\backslash Y=(\beta X\backslash X)\backslash\tau^{-1}_Y(p)=(\beta X\backslash X)\backslash Z(g)=g^{-1}\big((0,1]\big)=\bigcup_{n=1}^\infty g^{-1}\big([1/n,1]\big)\]
and  each  set $g^{-1}([1/n,1])$, for $n\in \mathbf{N}$, is compact (as it is closed in $\beta X\backslash X$ and the latter is compact, as $X$  is locally compact) and thus is closed in $\beta Y$.  Therefore, $\beta Y\backslash Y$ is an $F_\sigma$ in $\beta Y$, i.e., $Y$ is \v{C}ech-complete.
\end{proof}

In what follows, the following subset $\lambda_{{\mathcal P}} X$ of $\beta X$ plays a crucial
role. As we will see (Lemma \ref{JIUH}), $\lambda_{{\mathcal P}} X$ takes on a more familiar
form in the case when ${\mathcal P}$ is pseudocompactness.

\begin{definition}\label{HGA}
{\em For a  Tychonoff
space $X$ and a topological property  ${\mathcal P}$, let
\[\lambda_{{\mathcal P}} X=\bigcup\big\{\mbox{int}_{\beta X} \mbox{cl}_{\beta X}C:C\in Coz(X)\mbox{ and } \mbox{cl}_X C \mbox{ has }{\mathcal P}\big\}.\]}
\end{definition}

\begin{remark}
In  \cite{Ko3}, for a Tychonoff space $X$ and a topological property $\mathcal{P}$, we have defined $\lambda_{\mathcal{P}} X$ to be the set
\[\bigcup\big\{\mbox{int}_{\beta X}\mbox{cl}_{\beta X}Z: Z\in {\mathscr Z}(X) \mbox { has $\mathcal{P}$}\big\}.\]
As we will see in Lemma \ref{FJHSG} the two definitions coincide for a closed hereditary topological properties ${\mathcal P}$.
\end{remark}

Recall that a  subset $A$ of a Tychonoff space $X$ is called {\em bounded} (or {\em relatively pseudocompact}), if every continuous
$f:X\rightarrow\mathbf{R}$  is bounded on $A$. The following has been  proved by K. Morita in \cite{M}. It has been  rediscovered by R.L. Blair and M.A. Swardson
in \cite{BS} (see the comment succeeding  Proposition 2.6 of \cite{BS}).

\begin{lemma}[Morita \cite{M}; Blair and Swardson \cite{BS}]\label{B}
Let $X$ be a Tychonoff space. For a subset  $A$ of  $X$ the following are equivalent:
\begin{itemize}
\item[\rm(1)] $A$ is bounded in $X$.
\item[\rm(2)] $\mbox{\em cl}_{\beta X} A\subseteq\upsilon X$.
\item[\rm(3)] $\mbox{\em cl}_{\upsilon X} A$ is compact.
\end{itemize}
\end{lemma}

The following result is due to  A.W. Hager and D.G. Johnson in \cite{HJ}. A direct proof may be found in \cite{C} (see also Theorem 11.24 of \cite{W}).

\begin{lemma}[Hager and Johnson \cite{HJ}]\label{A}
Let $X$ be a Tychonoff space. For an open  subset  $U$ of $X$, if  $\mbox{\em cl}_{\upsilon X} U$ is compact, then  $\mbox{\em cl}_X U$ is pseudocompact.
\end{lemma}

If $A$ is a dense subset of a space $X$ and $U$ is an open subset of $X$ then $\mbox{cl}_XU=\mbox{cl}_X(U\cap A)$, and thus, in particular, $U\subseteq\mbox{int}_X\mbox{cl}_X(U\cap A)$. The following simple observation will have numerous applications in future. We record it here for the sake of completeness.

\begin{lemma}\label{HYFU}
Let $X$ be a Tychonoff space, let $F:\beta X\rightarrow\mathbf{I}$ be continuous and let $r\in (0,1)$. Denote $f=F|X$. Then
\[F^{-1}\big([0,r)\big)\subseteq \mbox{\em int}_{\beta X}\mbox{\em cl}_{\beta X}\big(X\cap F^{-1}\big([0,r)\big)\big)=\mbox{\em int}_{\beta X}\mbox{\em cl}_{\beta X}f^{-1}\big([0,r)\big).\]
\end{lemma}

\begin{lemma}\label{JIUH}
Let $X$ be a Tychonoff  space and let ${\mathcal P}$ be pseudocompactness. Then
\[\lambda_{{\mathcal P}} X=\mbox{\em int}_{\beta X}\upsilon X.\]
\end{lemma}

\begin{proof}
Let $C\in Coz(X)$ be such that $\mbox{cl}_X C$ is pseudocompact. Then $\mbox{cl}_X C$ is bounded, and thus by Lemma \ref{B} we have
$\mbox{cl}_{\beta X} C\subseteq\upsilon X$. Therefore  $\mbox{int}_{\beta X}\mbox{cl}_{\beta X} C\subseteq\mbox{int}_{\beta X}\upsilon X$. This shows that $\lambda_{{\mathcal P}} X\subseteq\mbox{int}_{\beta X}\upsilon X$.

To show the reverse inclusion, let $t\in \mbox{int}_{\beta X}\upsilon X$. Let $f:\beta X\rightarrow\mathbf{I}$ be  continuous and such that $f(t)=0$
and $f|(\beta X\backslash \mbox{int}_{\beta X}\upsilon X)\equiv 1$. Let
\[D=X\cap f^{-1}\big([0,1/2)\big)\in Coz(X).\]
Note that
\[\mbox{cl}_{\beta X} D=\mbox{cl}_{\beta X} \big(X\cap f^{-1}\big([0,1/2)\big)\big)\subseteq f^{-1}\big([0,1/2]\big)\subseteq\mbox{int}_{\beta X}\upsilon X\subseteq\upsilon X.\]
Therefore, $\mbox{cl}_{\upsilon X} D=\mbox{cl}_{\beta X} D\cap\upsilon X=\mbox{cl}_{\beta X} D$ is compact, and thus by Lemma \ref{A} the space   $\mbox{cl}_X D$ is pseudocompact. By definition of
$\lambda_{{\mathcal P}} X$ and using Lemma \ref{HYFU}, we have
\[t\in f^{-1}\big([0,1/2)\big)\subseteq\mbox{int}_{\beta X}\mbox{cl}_{\beta X}\big(X\cap f^{-1}\big([0,1/2)\big)\big)=\mbox{int}_{\beta X}\mbox{cl}_{\beta X} D\subseteq\lambda_{{\mathcal P}} X.\]
This shows that $\mbox{int}_{\beta X}\upsilon X\subseteq\lambda_{{\mathcal P}} X$.
\end{proof}

The following is a slight modification of a lemma from \cite{Ko3}.

\begin{lemma}\label{HGKFH}
Let $X$ be a Tychonoff  space and let ${\mathcal P}$ be a closed hereditary topological property which is preserved under  finite closed sums of subspaces. For a subset  $A$ of $X$, if $\mbox{\em cl}_{\beta X}A\subseteq\lambda_{{\mathcal P}} X$, then $\mbox{\em cl}_X A$ has ${\mathcal P}$.
\end{lemma}

\begin{proof}
By compactness and definition of
$\lambda_{{\mathcal P}} X$, we have
\[\mbox{cl}_{\beta X}A\subseteq\bigcup_{i=1}^n\mbox{int}_{\beta X}\mbox{cl}_{\beta X}C_i\]
for some $C_1,\ldots,C_n\in Coz(X)$, where $\mbox{cl}_X C_i$ has  ${\mathcal P}$ for $i=1,\ldots,n$.
Now $\mbox{cl}_X A\subseteq\bigcup_{i=1}^n\mbox{cl}_X C_i$, and the latter has ${\mathcal P}$, as it is a finite union of its closed  ${\mathcal P}$-subspaces, thus its closed subset $\mbox{cl}_X A$ also has ${\mathcal P}$.
\end{proof}

The following  follows from Theorem 3.2 and Corollary 3.3 of \cite{Ha}.

\begin{lemma}[Harris \cite{Ha}]\label{HGJHG}
Let $X$ be a Tychonoff space. Then $X\subseteq\mbox{\em int}_{\beta X}\upsilon X$ if and only if $X$ is locally pseudocompact.
\end{lemma}

The following is a slight modification of a lemma from \cite{Ko3}.

\begin{lemma}\label{JHG}
Let $X$ be a Tychonoff space and let $\mathcal{P}$ be a closed hereditary topological property.  Then $X\subseteq\lambda_{{\mathcal P}} X$ if
and only if $X$ is locally-$\mathcal{P}$.
\end{lemma}

\begin{proof}
Suppose that $X$ is  locally-$\mathcal{P}$. Let $x\in X$, and let $U$ be an open neighborhood of $x$ in $X$ whose closure $\mbox{cl}_XU$ has
$\mathcal{P}$. Let $f:X\rightarrow \mathbf{I}$ be continuous and such that  $f(x)=0$ and $f|(X\backslash U)\equiv 1$. Let
$f_\beta:\beta X\rightarrow \mathbf{I}$ continuously extend $f$. Let $C=f^{-1}([0,1/2))\in Coz(X)$. Note that $C\subseteq U$ and thus $\mbox{cl}_XC$  has $\mathcal{P}$, as it is closed in $\mbox{cl}_XU$. By definition of
$\lambda_{{\mathcal P}} X$ and using Lemma \ref{HYFU} we have
\[x\in f_\beta^{-1}\big([0,1/2)\big)\subseteq\mbox{int}_{\beta X}\mbox{cl}_{\beta X}f^{-1}\big([0,1/2)\big)=\mbox{int}_{\beta X}\mbox{cl}_{\beta X}C\subseteq\lambda_{{\mathcal P}} X.\]
This shows that $X\subseteq\lambda_{{\mathcal P}} X$.

To show the  converse, suppose that  $X\subseteq\lambda_{{\mathcal P}} X$. Let $x\in X$. Then $x\in \lambda_{{\mathcal P}} X$, and thus, by definition of
$\lambda_{{\mathcal P}} X$, there is a $D\in Coz(X)$ such that $x\in \mbox{int}_{\beta X}\mbox{cl}_{\beta X}D$ and $\mbox{cl}_X D$ has $\mathcal{P}$.
Let $V=X\cap \mbox{int}_{\beta X}\mbox{cl}_{\beta X}D$. Then $V$ is an open neighborhood of $x$ in $X$ whose closure $\mbox{cl}_XV$ has $\mathcal{P}$, as
\[\mbox{cl}_XV= X\cap\mbox{cl}_{\beta X}V\subseteq X\cap\mbox{cl}_{\beta X}D=\mbox{cl}_X D.\]
Therefore $X$ is locally-$\mathcal{P}$.
\end{proof}

\begin{definition}\label{HGFB}
{\em A  topological
property  $\mathcal{P}$ is said to {\em satisfy  Mr\'{o}wka's condition} (W) (or, to {\em be a Mr\'{o}wka topological property}), if it satisfies the following: If $X$ is a Tychonoff
space in which there is a point  $p$ with an open  base  ${\mathscr B}$ at $p$ such that $X\backslash  B$ has $\mathcal{P}$ for
each $B\in{\mathscr B}$, then $X$ has $\mathcal{P}$ (see \cite{Mr}).}
\end{definition}

\begin{remark}
If $\mathcal{P}$ is  a topological property which is closed hereditary and productive, then  Mr\'{o}wka's condition (W) is equivalent to
the following condition:   If a Tychonoff  space $X$ is the union of a compact space and a space with $\mathcal{P}$,
then $X$ has $\mathcal{P}$ (see \cite{MRW1}).
\end{remark}

The following example  provides a list of  Mr\'{o}wka topological properties (see \cite{Bu}, \cite{Steph} and \cite{Va} for definitions).

\begin{example}\label{JHL}
Consider the following topological properties. (1) The  Lindel\"{o}f property (2) Paracompactness (3) Metacompactness (4) Subparacompactness
(5) The para-Lindel\"{o}f property (6) The  $\sigma$-para-Lindel\"{o}f  property  (7)  Weak $\theta$-refinability (8)  $\theta$-refinability (or  submetacompactness) (9) Weak $\delta\theta$-refinability (10)  $\delta\theta$-refinability (or the submeta-Lindel\"{o}f property).  Let  ${\mathcal P}=\mbox{regularity}+(i)$ for $i=1,\ldots,10$. Then $\mathcal{P}$ is a Mr\'{o}wka topological property (see \cite{Ko3}).
\end{example}

In addition to the above topological properties, the list of Mr\'{o}wka topological properties includes: countable paracompactness, $[\theta,\kappa]$-compactness,  $\kappa$-boundedness,  screenability,  $\sigma$-metacompactness,  Dieudonn\'{e} completeness, $N$-compactness (\cite{M1}),  realcompactness, almost realcompactness (\cite{F}) and zero-dimensionality (see \cite{Ko3}, \cite{MRW1} and \cite{MRW2} for details).

Part of the next lemma is a simplified version of a lemma we have proved in \cite{Ko3}.

\begin{lemma}\label{JHKFH}
Let $X$ be a Tychonoff  space and let ${\mathcal P}$ be either pseudocompactness, or a closed hereditary Mr\'{o}wka topological property, which is preserved under finite closed sums of subspaces. For $Y\in{\mathscr E}(X)$ the following are equivalent:
\begin{itemize}
\item[\rm(1)] $Y\in{\mathscr E}_{{\mathcal P}}(X)$.
\item[\rm(2)] $X$ is locally-${\mathcal P}$ and $\beta X\backslash\lambda_{{\mathcal P}} X\subseteq\Theta_X(Y)$.
\end{itemize}
Thus, in particular
\[\Theta_X\big({\mathscr E}_{{\mathcal P}}(X)\big)=\{C\subseteq\beta X\backslash X:C \mbox{ is   compact and }\beta X\backslash\lambda_{{\mathcal P}} X\subseteq C\}\backslash\{\emptyset\}.\]
\end{lemma}

\begin{proof}
{\em Let ${\mathcal P}$ be  pseudocompactness.} Note that by Lemma \ref{JIUH} we have
\[\beta X\backslash\lambda_{{\mathcal P}} X=\beta X\backslash\mbox{int}_{\beta X}\upsilon X=\mbox{cl}_{\beta X}(\beta X\backslash\upsilon X).\]
Thus in this case, condition  (2) is equivalent to the requirement that $X$ is locally pseudocompact and
$\beta X\backslash\upsilon X\subseteq\tau^{-1}_Y(Y\backslash X)$.

Let $Y=X\cup\{p\}$. (1) {\em  implies} (2). Suppose to the contrary that $\beta X\backslash\upsilon X\nsubseteq\tau^{-1}_Y(p)$.
Let $t\in \beta X\backslash\upsilon X$ be such that $t\notin \tau^{-1}_Y(p)$. Let $S\in {\mathscr Z}(\beta X)$ be such that $t\in S$ and
$S\cap \tau^{-1}_Y(p)=\emptyset$. Since $t\notin \upsilon X$, there is a  $Z\in{\mathscr Z}(\beta X)$ such that $t\in Z$ and $X\cap Z=\emptyset$.
Note that by Lemma \ref{JFV}, the space $\beta Y$ coincides with the quotient space of $\beta X$ which is obtained by contracting $\tau^{-1}_Y(p)$ to a point with $\tau_Y$ as its quotient mapping.  Now, if we let $G=(S\cap Z)\backslash\tau^{-1}_Y(p)$, then $\tau_Y(G)$ is a non-empty $G_\delta$ in  $\beta Y$ which misses $Y$, contradicting the pseudocompactness of $Y$. Note that pseudocompactness is hereditary with respect to regular-closed subsets. Thus $X$, having a pseudocompact extension with a one-point remainder, is locally pseudocompact.

(2) {\em  implies} (1).   Suppose to the contrary that $Y$ is not pseudocompact. Then there is a non-empty
$T\in {\mathscr Z}(\beta Y)$ such that $Y\cap T=\emptyset$. But then $\tau^{-1}_Y(T)\in  {\mathscr Z}(\beta X)$ is non-empty and misses $X$, which implies that $\tau^{-1}_Y(T)\subseteq\beta X\backslash\upsilon X$. Therefore $\tau^{-1}_Y(T)\subseteq \tau^{-1}_Y(p)$,  which is a contradiction,
as $p\notin T$.

{\em Let ${\mathcal P}$ be a closed hereditary Mr\'{o}wka topological property, which is preserved under finite closed sums of subspaces.}

Let $Y=X\cup\{p\}$. (1) {\em  implies} (2).
Suppose to the contrary that  $t\notin \tau^{-1}_Y(p)$ for some $t\in \beta X\backslash\lambda_{{\mathcal P}} X$.  Let $f:\beta X\rightarrow \mathbf{I}$
be continuous and such that $f(t)=0$ and $f|\tau^{-1}_Y(p)\equiv 1$. Since  $\tau_Y(f^{-1}([0,1/2]))$ is compact (as it is a continuous image of a compact space) and thus closed in $\beta Y$, the set
\[T=X\cap f^{-1}\big([0,1/2]\big)=Y\cap\tau_Y\big(f^{-1}\big([0,1/2]\big)\big)\]
is closed in $Y$, and therefore it has ${\mathcal P}$. Let $C=X\cap f^{-1}([0,1/2))$. Then $C\in Coz(X)$ and $\mbox{cl}_X C$ has ${\mathcal P}$, as it is  closed in $T$.
By definition of $\lambda_{{\mathcal P}} X$ and using Lemma \ref{HYFU} we have
\[t\in f^{-1}\big([0,1/2)\big)\subseteq\mbox{int}_{\beta X}\mbox{cl}_{\beta X}\big(X\cap f^{-1}\big([0,1/2)\big)\big)=\mbox{int}_{\beta X}\mbox{cl}_{\beta X}C\subseteq\lambda_{{\mathcal P}} X\]
which is a contradiction. This shows that  $\beta X\backslash\lambda_{{\mathcal P}} X\subseteq\tau^{-1}_Y(p)$.
Obviously, since ${\mathcal P}$ is a  closed hereditary topological property and $X$ has a ${\mathcal P}$-extension with a one-point remainder, $X$
is  locally-${\mathcal P}$.

(2) {\em  implies} (1).  Since ${\mathcal P}$  satisfies Mr\'{o}wka's condition (W), to show that $Y$ has ${\mathcal P}$, it suffices to verify that $Y\backslash V$ has  ${\mathcal P}$, for every open neighborhood $V$ of $p$ in $Y$. Let $V$ be an
open neighborhood of $p$ in $Y$. Let $V'$ be an open subset of $\beta Y$ such that $Y\cap V'=V$. Then since
\[Y\backslash V=X\backslash V\subseteq\beta X\backslash\tau^{-1}_Y(V')\subseteq\beta X\backslash\tau^{-1}_Y(p)\subseteq\lambda_{{\mathcal P}} X\]
we have
\[\mbox{cl}_{\beta X}(Y\backslash V)\subseteq\beta X\backslash\tau^{-1}_Y(V')\subseteq\lambda_{{\mathcal P}} X\]
and thus by Lemma \ref{HGKFH}, the set $Y\backslash V$  has ${\mathcal P}$.

The final remark follows  from above and Theorem \ref{JHV} in the case  when ${\mathscr E}_{{\mathcal P}}(X)\neq\emptyset$, as in this case (by above) $X$ is locally-${\mathcal P}$.  In the case when  ${\mathscr E}_{{\mathcal P}}(X)=\emptyset$, note that the existence of a non-empty compact
$C\subseteq\beta X\backslash X$ such that $\beta X\backslash\lambda_{{\mathcal P}} X\subseteq C$ implies that $\beta X\backslash\lambda_{{\mathcal P}} X\subseteq \beta X\backslash X$, or equivalently $X\subseteq\lambda_{{\mathcal P}} X$. But by Lemmas
\ref{HGJHG} and  \ref{JHG} it then follows that $X$ is locally-$\mathcal{P}$, and therefore $\Theta_X^{-1}(C)\in{\mathscr E}_{{\mathcal P}}(X)$, which is a contradiction.
\end{proof}

The following lemma is motivated by Lemma 3.11 of \cite{Ko2}.

\begin{lemma}\label{SGFH}
Let $X$ be a Tychonoff space and let  ${\mathcal P}$ be either pseudocompactness or a  closed hereditary  topological property,
which is preserved under finite closed sums of subspaces. For $Y\in{\mathscr E}(X)$ the following are equivalent:
\begin{itemize}
\item[\rm(1)] $Y\in{\mathscr E}_{{\mathcal P}-far}(X)$.
\item[\rm(2)] $\Theta_X(Y)\subseteq\beta X\backslash\lambda_{{\mathcal P}} X$.
\end{itemize}
\end{lemma}

\begin{proof}
Let $Y=X\cup\{p\}$. (1) {\em  implies} (2).  Suppose to the contrary that $t\notin \beta X\backslash\lambda_{{\mathcal P}} X$
for some $t\in \tau_Y^{-1} (p)$. Note that by definition the set $\lambda_{{\mathcal P}} X$ is open in $\beta X$. Let $f:\beta X\rightarrow {\mathbf I}$ be continuous
and such that $f(t)=0$ and $f|( \beta X\backslash\lambda_{{\mathcal P}} X)\equiv 1$. Let
\[Z= X\cap f^{-1}\big([0,1/3]\big)\in {\mathscr Z}(X)\mbox{ and } C=X\cap f^{-1}\big([0,1/2)\big)\in Coz(X).\]
Note that $C\subseteq f^{-1}([0,1/2])$, and since $\mbox{cl}_{\beta X}C\subseteq f^{-1}([0,1/2])\subseteq\lambda_{{\mathcal P}} X$,
by Lemmas  \ref{B}, \ref{A} and \ref{HGKFH} it follows that $\mbox{cl}_XC$  has ${\mathcal P}$.
Therefore, using the assumption  we have $\mbox{cl}_Y Z\cap \{p\}=\emptyset$, or equivalently $p\notin \mbox{cl}_Y Z=Y\cap \mbox{cl}_{\beta Y} Z$, or $p\notin \mbox{cl}_{\beta Y} Z$. But this implies that $\tau_Y^{-1} (p)\cap\tau_Y^{-1} (\mbox{cl}_{\beta Y} Z)=\emptyset$, and thus $t\notin \tau_Y^{-1} (\mbox{cl}_{\beta Y} Z)$, as
$t\in \tau_Y^{-1} (p)$. Now
\begin{eqnarray*}
X\cap f^{-1}\big([0,1/3)\big)&\subseteq&\tau_Y^{-1} \big(X\cap f^{-1}\big([0,1/3)\big)\big)\\&\subseteq&\tau_Y^{-1}\big(X\cap f^{-1}\big([0,1/3]\big)\big)=\tau_Y^{-1} (Z)\subseteq\tau_Y^{-1}(\mbox{cl}_{\beta Y} Z)
\end{eqnarray*}
and the latter set is closed in $\beta X$, therefore
\[t\in f^{-1}\big([0,1/3)\big)\subseteq\mbox{cl}_{\beta X}f^{-1}\big([0,1/3)\big)=\mbox{cl}_{\beta X}\big(X\cap f^{-1}\big([0,1/3)\big)\big)\subseteq\tau_Y^{-1} (\mbox{cl}_{\beta Y} Z).\]
This contradiction proves that $\tau_Y^{-1} (p)\subseteq \beta X\backslash\lambda_{{\mathcal P}} X$.

(2) {\em  implies} (1).   Let $Z\in {\mathscr Z}(X)$ be such that $Z\subseteq C$, for some $C\in Coz(X)$ whose closure
$\mbox{cl}_X C$ has ${\mathcal P}$. The zero-sets $Z$ and $X\backslash C$ of $X$ are  disjoint, and thus completely separated in $X$. Let
$g:X\rightarrow\mathbf{I}$ be continuous and such that $g|Z \equiv 0$ and  $g|(X\backslash C)\equiv 1$. Let
$g_\beta:\beta X\rightarrow\mathbf{I}$  continuously  extend $g$. Let
\[D=g^{-1}\big([0,1/2)\big)\in Coz(X).\]
Then $D\subseteq C$,
and  since  $\mbox{cl}_X D$ is regular-closed in $\mbox{cl}_X C$, the set $\mbox{cl}_X D$ has  $\mathcal{P}$.
(Note again that pseudocompactness is hereditary with respect to regular-closed subsets.)  By definition of $\lambda_{{\mathcal P}} X$ and using Lemma \ref{HYFU} we have
\begin{eqnarray*}
g_\beta^{-1}\big([0,1/2)\big)&\subseteq&\mbox{int}_{\beta X}\mbox{cl}_{\beta X}\big(X\cap g_\beta^{-1}\big([0,1/2)\big)\big)\\&=&\mbox{int}_{\beta X}\mbox{cl}_{\beta X}g^{-1}\big([0,1/2)\big)=\mbox{int}_{\beta X}\mbox{cl}_{\beta X}D\subseteq\lambda_{{\mathcal P}} X.
\end{eqnarray*}
But
\begin{eqnarray*}
\mbox{cl}_{\beta X}Z\subseteq \mbox{cl}_{\beta X}g^{-1}\big([0,1/3)\big)&=&\mbox{cl}_{\beta X}\big(X\cap g_\beta^{-1}\big([0,1/3)\big)\big)\\&=&\mbox{cl}_{\beta X}g_\beta^{-1}\big([0,1/3)\big)\subseteq g_\beta^{-1}\big([0,1/3]\big)
\end{eqnarray*}
which together with above implies that $\mbox{cl}_{\beta X}Z\subseteq\lambda_{{\mathcal P}} X$.
Recall that by Lemma \ref{JFV}, the space $\beta Y$ is the quotient space of $\beta X$ which is obtained by contracting  $\tau_Y^{-1}(p)$ to $p$ and $\tau_Y$ is the quotient
mapping. Note that by assumption we have $\tau_Y^{-1} (p)\subseteq \beta X\backslash\lambda_{{\mathcal P}} X$. Now $Z\subseteq\tau_Y (\mbox{cl}_{\beta X} Z)$ and the latter is closed in $\beta Y$ (as it is compact, as it is a continuous image of a compact space), therefore
\[\mbox{cl}_{\beta Y} Z \subseteq\tau_Y(\mbox{cl}_{\beta X} Z)=\mbox{cl}_{\beta X} Z\subseteq\lambda_{{\mathcal P}} X\]
and  since
\[\mbox{cl}_Y Z\cap \{p\}\subseteq\mbox{cl}_{\beta Y} Z\cap \{p\}\subseteq\lambda_{{\mathcal P}} X\cap \{p\}=\emptyset\]
we have $\mbox{cl}_Y Z\cap \{p\}=\emptyset$. This shows that  $Y$ is a ${\mathcal P}$-far extension of $X$.
\end{proof}

\section{Embedding the set of  one-point first-countable  extensions into the set of one-point ${\mathcal P}$-extensions}

In this section, we anti-order-isomorphically embed the set of one-point first-countable  extensions of a space  into the set of its one-point ${\mathcal P}$-extensions.

\begin{notation}\label{KKJB}
{\em Let $X$ be a Tychonoff space. For a subset $A$ of $X$, let
\[A^\star=\mbox{cl}_{\beta X}A\backslash X.\]
Thus, in particular,  $X^\star=\beta X\backslash X$.}
\end{notation}

\begin{remark}\label{HFS}
The notation given in Notation \ref{KKJB} can be ambiguous, as in the case when $A$ is not $C^*$-embedded in $X$, the notation $A^\star$ can mean either $\beta A\backslash A$ or $\mbox{cl}_{\beta X}A\backslash X$ and these spaces are not homeomorphic. In such situations, we will always take $A^\star$ to mean the second of these two possibilities.
\end{remark}

The following is Lemma 4.6 of \cite{Ko2}.

\begin{lemma}\label{HSJHG}
Let $X$ be a  locally compact space. If a $Z\in {\mathscr Z}(\beta X)$  misses $X$, then $Z$ is regular-closed in $X^\star$.
\end{lemma}

\begin{proof}
Let $x\in Z$. Suppose that  $x\notin \mbox{cl}_{X^\star}\mbox{int}_{X^\star} Z$. Let $S\in {\mathscr Z}(\beta X)$ be such that $x\in S$ and  $S\cap  \mbox{cl}_{X^\star}\mbox{int}_{X^\star} Z=\emptyset$. Let $T=S\cap Z$. Then  $\mbox{int}_{X^\star} T\neq\emptyset$. (Recall that for any locally compact space $Y$, any non-empty zero-set of $\beta Y$ which is contained in  $Y^\star$ has a non-empty interior in $Y^\star$; see Lemma 15.17 of \cite{CN}.) But this is a contradiction, as
$\mbox{int}_{X^\star} T\subseteq\mbox{int}_{X^\star} Z$ and $T\cap\mbox{int}_{X^\star} Z=\emptyset$, as $T\cap\mbox{int}_{X^\star} Z\subseteq S\cap\mbox{cl}_{X^\star} \mbox{int}_{X^\star} Z$. Thus
$x\in \mbox{cl}_{X^\star}\mbox{int}_{X^\star} Z$, which shows that $Z$ is regular-closed in $X^\star$.
\end{proof}

\begin{lemma}\label{KLB}
Let $X$ be a Tychonoff space. Then for every $t\in X^\star$ we have
\[\bigcap\big\{T^\star:T\in {\mathscr Z}(X)\mbox{ and }t\in\mbox{\em int}_{\beta X}\mbox{\em cl}_{\beta X}T\big\}=\{t\}.\]
\end{lemma}

\begin{proof}
Let $t\in X^\star$. Obviously, $t\in T^\star$ for every $T\in {\mathscr Z}(X)$ such that $t\in\mbox{int}_{\beta X}\mbox{cl}_{\beta X}T$.
To show the converse, suppose to the contrary that there is a  $s\neq t$  such that $s\in T^\star$, for every $T\in {\mathscr Z}(X)$ with
$t\in\mbox{int}_{\beta X}\mbox{cl}_{\beta X}T$. Let  $f:\beta X\rightarrow{\mathbf I}$ be continuous and such that $f(t)=0$ and
$f(s)=1$. Let
\[T=X\cap f^{-1}\big([0,1/2]\big)\in {\mathscr Z}(X).\]
Note that using Lemma \ref{HYFU} we have
\begin{eqnarray*}
t\in f^{-1}\big([0,1/2)\big)&\subseteq&\mbox{int}_{\beta X}\mbox{cl}_{\beta X}\big(X\cap f^{-1}\big([0,1/2)\big)\big)\\&\subseteq&\mbox{int}_{\beta X}\mbox{cl}_{\beta X}\big(X\cap f^{-1}\big([0,1/2]\big)\big)=\mbox{int}_{\beta X}\mbox{cl}_{\beta X}T.
\end{eqnarray*}
But $s\notin \mbox{cl}_{\beta X}T$, as  $\mbox{cl}_{\beta X}T\subseteq f^{-1}([0,1/2])$.
\end{proof}

The following is a slight modification of a lemma from \cite{Ko3}.

\begin{lemma}\label{KJHG}
Let $X$ be a Tychonoff locally-$\mathcal{P}$ space, where  $\mathcal{P}$ is  a closed hereditary topological property, which  is preserved under  finite
closed sums of subspaces.  Then $\lambda_{{\mathcal P}} X$ is compact if and only if $X$ has $\mathcal{P}$.
\end{lemma}

\begin{proof}
Suppose that $\lambda_{{\mathcal P}} X$ is compact. By Lemma \ref{JHG} we have $X\subseteq\lambda_{{\mathcal P}} X$, and thus $\mbox{cl}_{\beta X}X\subseteq\lambda_{{\mathcal P}} X$. But by Lemma \ref{HGKFH}, this implies that $X$ has $\mathcal{P}$. The converse follows from definition of $\lambda_{{\mathcal P}} X$ (and the  obvious fact that  $X\in Coz(X)$).
\end{proof}

\begin{lemma}\label{FJHSG}
Let $X$ be a Tychonoff space and let $\mathcal{P}$ be a  closed hereditary topological property. Then
\[\lambda_{{\mathcal P}} X=\bigcup\big\{\mbox{\em int}_{\beta X} \mbox{\em cl}_{\beta X}Z:Z\in {\mathscr Z}(X)\mbox{ has } {\mathcal P}\big\}.\]
\end{lemma}

\begin{proof}
Let $t\in \lambda_{{\mathcal P}} X$. Then $t\in \mbox{int}_{\beta X} \mbox{cl}_{\beta X}C$ for some  $C\in Coz(X)$ such that $\mbox{cl}_X C$ has
${\mathcal P}$. Let $f:\beta X\rightarrow{\mathbf I}$ be continuous and such that $f(t)=0$ and
$f|(\beta X\backslash\mbox{int}_{\beta X} \mbox{cl}_{\beta X}C)\equiv 1$. Let
\[Z=X\cap f^{-1}\big([0,1/2]\big)\in {\mathscr Z}(X).\]
Since
\[f^{-1}\big([0,1/2]\big)\subseteq\mbox{int}_{\beta X} \mbox{cl}_{\beta X}C\subseteq \mbox{cl}_{\beta X}C\]
we have
\[Z=X\cap f^{-1}\big([0,1/2]\big)\subseteq X \cap\mbox{cl}_{\beta X}C=\mbox{cl}_X C\]
and thus $Z$ has  $\mathcal{P}$, as it is closed in $\mbox{cl}_X C$.  Using Lemma \ref{HYFU} we have
\begin{eqnarray*}
t\in f^{-1}\big([0,1/2)\big)&\subseteq&\mbox{int}_{\beta X}\mbox{cl}_{\beta X}\big(X\cap f^{-1}\big([0,1/2)\big)\big)\\&\subseteq&\mbox{int}_{\beta X}\mbox{cl}_{\beta X}\big(X\cap f^{-1}\big([0,1/2]\big)\big)=\mbox{int}_{\beta X}\mbox{cl}_{\beta X}Z.
\end{eqnarray*}

Next, to show the reverse inclusion, let  $s\in \mbox{int}_{\beta X}\mbox{cl}_{\beta X}S$ for some $S\in {\mathscr Z}(X)$ which has  $\mathcal{P}$. Let
$g:\beta X\rightarrow{\mathbf I}$ be continuous and such that $g(s)=0$ and $g|(\beta X\backslash\mbox{int}_{\beta X} \mbox{cl}_{\beta X}S)\equiv 1$.
Let
\[D=X\cap g^{-1}\big([0,1/2)\big)\in Coz(X).\]
Note that
\[D=X\cap g^{-1}\big([0,1/2)\big)\subseteq X\cap\mbox{int}_{\beta X}\mbox{cl}_{\beta X}S\subseteq X\cap\mbox{cl}_{\beta X}S=S\]
and therefore $\mbox{cl}_X D$ has  $\mathcal{P}$, as it is closed in $S$. By definition of
$\lambda_{{\mathcal P}} X$ and using Lemma \ref{HYFU} we have
\[s\in g^{-1}\big([0,1/2)\big)\subseteq\mbox{int}_{\beta X}\mbox{cl}_{\beta X}\big(X\cap g^{-1}\big([0,1/2)\big)\big)=\mbox{int}_{\beta X}\mbox{cl}_{\beta X}D\subseteq\lambda_{{\mathcal P}} X.\]
\end{proof}

The following generalizes  Lemma 4.5 of \cite{Ko2}.  Lemma 4.5 of \cite{Ko2}, itself, exploits  an idea from Lemma 6.4 of \cite{HJW}.

\begin{lemma}\label{JDSB}
Let $X$ be a locally compact locally-${\mathcal P}$ space, where ${\mathcal P}$ is a  closed hereditary topological property. Suppose that either
\begin{itemize}
\item[\rm(1)] $X$ is paracompact and ${\mathcal P}$ is preserved under locally finite closed sums of subspaces, or
\item[\rm(2)] ${\mathcal P}$ is preserved under countable closed sums of subspaces.
\end{itemize}
If a $Z\in {\mathscr Z}(\beta X)$ is such that $X\cap Z=\emptyset$, then $\mbox{\em int}_{X^\star}Z\subseteq \lambda_{{\mathcal P}} X$.
\end{lemma}

\begin{proof}
Let $t\in \mbox{int}_{X^\star}Z$. Since by Lemma  \ref{KLB} we have
\[\bigcap\big\{T^\star:T\in {\mathscr Z}(X)\mbox{ and }t\in\mbox{int}_{\beta X}\mbox{cl}_{\beta X}T\big\}=\{t\}\subseteq\mbox{int}_{X^\star}Z\]
by compactness of $X^\star$ (as  $X$ is locally compact),  it follows that
\[\bigcap_{i=1}^n T^\star_i\subseteq\mbox{int}_{X^\star}Z\]
for some  $T_1,\ldots,T_n\in {\mathscr Z}(X)$ such that $t\in \mbox{int}_{\beta X}\mbox{cl}_{\beta X}T_i$ for $i=1,\ldots,n$.
Then
\begin{equation}\label{KJUJB}
t\in \bigcap_{i=1}^n\mbox{int}_{\beta X}\mbox{cl}_{\beta X}T_i=\mbox{int}_{\beta X}\Big(\bigcap_{i=1}^n\mbox{cl}_{\beta X}T_i\Big)=\mbox{int}_{\beta X}\mbox{cl}_{\beta X}\Big(\bigcap_{i=1}^nT_i\Big)=\mbox{int}_{\beta X}\mbox{cl}_{\beta X}T
\end{equation}
where
\[T=\bigcap_{i=1}^n T_i\in {\mathscr Z}(X).\]
Note that
\begin{eqnarray*}
\mbox{cl}_{\beta X}T\backslash X=\mbox{cl}_{\beta X}\Big(\bigcap_{i=1}^nT_i\Big)\backslash X&=&\Big(\bigcap_{i=1}^n\mbox{cl}_{\beta X}T_i\Big)\backslash X\\&=&\bigcap_{i=1}^n(\mbox{cl}_{\beta X}T_i\backslash X)=\bigcap_{i=1}^nT^\star_i\subseteq\mbox{int}_{X^\star}Z\subseteq Z.
\end{eqnarray*}
Therefore $\mbox{cl}_{\beta X}T\backslash Z\subseteq X$. Let $Z=Z(f)$ for some continuous $f:\beta X\rightarrow\mathbf{I}$. Let $k\in \mathbf{N}$. Then
\begin{equation}\label{JUJB}
\mbox{cl}_{\beta X}T\backslash f^{-1}\big([0,1/k)\big)\subseteq \mbox{cl}_{\beta X}T\backslash Z \subseteq X.
\end{equation}
Since $X$ is  locally-${\mathcal P}$, for each $x\in X$ there is an open neighborhood $U_x$ of $x$ in $X$ such that $\mbox{cl}_X U_x$
has ${\mathcal P}$. By compactness, from (\ref{JUJB}) it follows that
\[\mbox{cl}_{\beta X}T\backslash f^{-1}\big([0,1/k)\big)\subseteq \bigcup_{i=1}^{j_k}U_{x_i^k}\]
for some $j_k\in \mathbf{N}$ and some $x^k_1,\ldots,x^k_{j_k}\in X$. Then
\begin{eqnarray*}
T\subseteq\mbox{cl}_{\beta X}T\backslash Z=\mbox{cl}_{\beta X}T\backslash Z(f)&=&\mbox{cl}_{\beta X}T\backslash \bigcap_{k=1}^{\infty}f^{-1}\big([0,1/k)\big)\\&=&\bigcup_{k=1}^{\infty}\big(\mbox{cl}_{\beta X}T\backslash f^{-1}\big([0,1/k)\big)\big)\subseteq\bigcup_{k=1}^{\infty}\bigcup_{i=1}^{j_k}U_{x_i^k}.
\end{eqnarray*}
We show that  $T$  has ${\mathcal P}$. Consider the following cases:

\begin{description}
\item[{\sc Case (1)}] {\em $X$ is paracompact and ${\mathcal P}$ is preserved under locally finite closed sums of subspaces.} The open cover
\[{\mathscr U}=\{X\backslash T\}\cup\{U_{x_i^k}:k\in {\mathbf N}\mbox{ and }i=1,\ldots,j_k\}\]
of $X$ has a locally finite closed refinement ${\mathscr F}$. Now, for each $F\in {\mathscr F}$ for which  $F\cap T\neq\emptyset$, we have
$F\subseteq U_{x_i^k}$ for some $k\in {\mathbf N}$  and some $i=1,\ldots,j_k$. Thus $F$ has ${\mathcal P}$, as it is closed in $\mbox{cl}_XU_{x_i^k}$. Therefore
\[G=\bigcup\{F\in {\mathscr F}:F\cap T\neq\emptyset\}\]
has  ${\mathcal P}$. Since $T$ is  closed in $G$, this now implies that $T$ has ${\mathcal P}$.
\item[{\sc Case (2)}] {\em ${\mathcal P}$ is preserved under countable closed sums of subspaces.} In this case, note that if we let \[G=\bigcup_{k=1}^{\infty}\bigcup_{i=1}^{j_k}\mbox{cl}_XU_{x_i^k}\]
then $G$ has ${\mathcal P}$, and thus, its closed subset $T$ also has ${\mathcal P}$.
\end{description}
Thus $T$ has ${\mathcal P}$ in either cases. By Lemma \ref{FJHSG}  we have $\mbox{int}_{\beta X}\mbox{cl}_{\beta X}T\subseteq\lambda_{{\mathcal P}} X$, and therefore by (\ref{KJUJB}), it follows that $t\in \lambda_{{\mathcal P}} X$. This shows that $\mbox{int}_{X^\star}Z\subseteq\lambda_{{\mathcal P}} X$.
\end{proof}

The following  generalizes Theorem 4.7 of  \cite{Ko2}.

\begin{theorem}\label{JHDSB}
Let $X$ be a locally compact locally-${\mathcal P}$ non-${\mathcal P}$  space, where ${\mathcal P}$ is a  closed hereditary Mr\'{o}wka topological property. Suppose that either
\begin{itemize}
\item[\rm(1)] $X$ is paracompact and ${\mathcal P}$ is preserved under locally finite closed sums of subspaces, or
\item[\rm(2)] ${\mathcal P}$ is preserved under countable closed sums of subspaces.
\end{itemize}
Then $({\mathscr E}_{{\mathcal P}}(X),\leq)$ contains an anti-order-isomorphic copy  of  $({\mathscr E}^*(X),\leq)$.
\end{theorem}

\begin{proof}
Let $Y\in{\mathscr E}^*(X)$. By Lemma \ref{JYYH}  we have $\Theta_X(Y) \in {\mathscr Z}(\beta X)$ and  $X\cap \Theta_X(Y)=\emptyset$. By Lemma
\ref{JDSB}, it then follows that  $\mbox{int}_{X^\star}\Theta_X(Y) \subseteq \lambda_{{\mathcal P}} X$. Now, by Lemma \ref{JHG} we have $X\subseteq \lambda_{{\mathcal P}} X$, as $X$ is
locally-${\mathcal P}$. Therefore
\[\beta X\backslash\lambda_{{\mathcal P}} X\subseteq X^\star\cap\big(\beta X\backslash\mbox{int}_{X^\star}\Theta_X(Y)\big)=X^\star\backslash\mbox{int}_{X^\star}\Theta_X(Y).\]
Note that the latter set is  compact (as it is closed in $X^\star$ and $X^\star$ is compact, as $X$ is locally compact) and non-empty  (as $\lambda_{{\mathcal P}} X\neq\beta X$, as $\lambda_{{\mathcal P}} X$ is non-compact, as $X$ is non-$\mathcal{P}$; see Lemma \ref{KJHG}). Thus, by Lemma \ref{JHKFH}, the function
\[F:\big({\mathscr E}^*(X),\leq\big)\rightarrow\big({\mathscr E}_{{\mathcal P}}(X),\leq\big)\]
defined by
\[F(Y)=\Theta_X^{-1}\big(X^\star\backslash\mbox{int}_{X^\star}\Theta_X(Y)\big)\]
for $Y\in {\mathscr E}^*(X)$ is well-defined. The fact that $F$ is an anti-order-homomorphism is straightforward. To complete the proof, note that if
\[\Theta_X^{-1}\big(X^\star\backslash\mbox{int}_{X^\star}\Theta_X(Y_1)\big)=F(Y_1)\leq F(Y_2)=\Theta_X^{-1}\big(X^\star\backslash\mbox{int}_{X^\star}\Theta_X(Y_2)\big)\]
for  $Y_1,Y_2\in{\mathscr E}^*(X)$, then (since $\Theta_X$ is an anti-order-isomorphism) we have
\[X^\star\backslash\mbox{int}_{X^\star}\Theta_X(Y_1)\supseteq X^\star\backslash\mbox{int}_{X^\star}\Theta_X(Y_2)\]
and therefore $\mbox{int}_{X^\star}\Theta_X(Y_1)\subseteq \mbox{int}_{X^\star}\Theta_X(Y_2)$. Using Lemma \ref{HSJHG}, it then follows that
\[\Theta_X(Y_1)=\mbox{cl}_{X^\star}\mbox{int}_{X^\star}\Theta_X(Y_1)\subseteq\mbox{cl}_{X^\star}\mbox{int}_{X^\star}\Theta_X(Y_2)=\Theta_X(Y_2)\]
which implies that $Y_1\geq Y_2$. Thus, $F$ is an  anti-order-isomorphism (onto its image).
\end{proof}

\begin{remark}\label{LJG}
The list of topological properties ${\mathcal P}$ satisfying the assumption of Theorem \ref{JHDSB} is quite  wide and include  all
topological properties (1)--(10) introduced in Example \ref{JHL}. Note that properties (1), (4), (7)--(10) are  preserved under countable closed sums of
subspaces and properties (2)--(6) and (8) are preserved under locally finite closed sums of subspaces (see Theorems 7.3 and 7.4 of \cite{Bu}).
Also, they  are all  hereditary with respect to  closed subsets (see Theorem 7.1 of \cite{Bu}).
\end{remark}

\section{Extensions and restrictions of order-isomorphisms between  sets of one-point ${\mathcal P}$-extensions}

In this section we consider order-isomorphisms between various sets of one-point ${\mathcal P}$-extensions. The results of this section are partly motivated by their applications in the next section.

\begin{lemma}\label{KJGT}
Let $X$ be a Tychonoff locally-${\mathcal P}$ non-${\mathcal P}$  space,  where  ${\mathcal P}$ is either pseudocompactness, or a
closed hereditary Mr\'{o}wka topological property, which is preserved under  finite closed sums of subspaces. Then
\begin{itemize}
\item[\rm(1)] $({\mathscr E}_{{\mathcal P}}(X),\leq)$ has a unique largest element, namely $M=\Theta_X^{-1}(\beta X\backslash\lambda_{{\mathcal P}} X)$.
\item[\rm(2)] For every non-empty collection $\{Y_i\}_{i\in I}\subseteq{\mathscr E}_{{\mathcal P}}(X)$, the least upper bound $\bigvee_{i\in I}Y_i$ exists in $({\mathscr E}_{{\mathcal P}}(X),\leq)$ and
    \[\bigvee_{i\in I}Y_i=\Theta_X^{-1}\Big(\bigcap_{i\in I}\Theta_X(Y_i)\Big).\]
\item[\rm(3)] Further, let $X$ be locally compact. Then for every non-empty countable collection $\{Y_i\}_{i\in I}\subseteq{\mathscr E}^C_{{\mathcal P}}(X)$, the least upper bound $\bigvee_{i\in I}Y_i$ exists in $({\mathscr E}^C_{{\mathcal P}}(X),\leq)$ and
    \[\bigvee_{i\in I}Y_i=\Theta_X^{-1}\Big(\bigcap_{i\in I}\Theta_X(Y_i)\Big).\]
\end{itemize}
\end{lemma}

\begin{proof}
These  easily follow from Lemmas  \ref{HDJYYH} and  \ref{JHKFH}. Note that since $X$ is locally-${\mathcal P}$ and non-${\mathcal P}$, the set $\lambda_{{\mathcal P}} X$
is non-compact (this follows from Lemmas \ref{JIUH}
and \ref{HGJHG} in the case  when  ${\mathcal P}$ is pseudocompactness, and in the other case, it follows from Lemma \ref{KJHG}) and thus
$\beta X\backslash\lambda_{{\mathcal P}} X\neq\emptyset$.
\end{proof}

\begin{notation}\label{HGTTDR}
{\em For a   space $X$  and a topological property ${\mathcal P}$, we denote by $M^X_{{\mathcal P}}$ the largest element of
$({\mathscr E}_{{\mathcal P}}(X),\leq)$ (provided it actually exists).}
\end{notation}

\begin{lemma}\label{HGTIITD}
Let $X$  be a locally compact locally-${\mathcal P}$ non-${\mathcal P}$  space,  where  ${\mathcal P}$ is either pseudocompactness, or a
closed hereditary Mr\'{o}wka topological property, which is preserved under  finite closed sums of subspaces.
For $Y\in{\mathscr E}_{{\mathcal P}}(X)$ the following are equivalent:
\begin{itemize}
\item[\rm(1)] $\Theta_X(Y)\cap\lambda_{{\mathcal P}} X$ is compact.
\item[\rm(2)] For every collection $\{Y_i\}_{i\in I}\subseteq{\mathscr E}_{{\mathcal P}}(X)$ such that $Y\vee\bigvee_{i\in I}Y_i=M^X_{{\mathcal P}}$, we have $Y\vee\bigvee_{j=1}^k Y_{i_j}=M^X_{{\mathcal P}}$ for some $i_1,\ldots,i_k\in I$.
\end{itemize}
\end{lemma}

\begin{proof}
(1) {\em  implies} (2). Let $\{Y_i\}_{i\in I}\subseteq{\mathscr E}_{{\mathcal P}}(X)$ be such
that $Y\vee\bigvee_{i\in I}Y_i=M^X_{{\mathcal P}}$. Then by Lemma \ref{KJGT} we have
\[\Theta_X(Y)\cap\bigcap_{i\in I}\Theta_X(Y_i)=\Theta_X\Big(Y\vee\bigvee_{i\in I}Y_i\Big)=\Theta_X(M^X_{{\mathcal P}})=\beta X\backslash\lambda_{{\mathcal P}} X\]
and thus
\[\lambda_{{\mathcal P}} X\cap\Theta_X(Y)\cap\bigcap_{i\in I}\Theta_X(Y_i)=\emptyset.\]
Now, by compactness of $\Theta_X(Y)\cap\lambda_{{\mathcal P}} X$ it follows that
\[\lambda_{{\mathcal P}} X\cap\Theta_X(Y)\cap\bigcap_{j=1}^k\Theta_X(Y_{i_j})=\emptyset\]
for some $i_1,\ldots,i_k\in I$. Therefore, using  Lemma \ref{KJGT} (note that by Lemma  \ref{JHKFH} each
element of $\Theta_X({\mathscr E}_{{\mathcal P}}(X))$ contains $\beta X\backslash\lambda_{{\mathcal P}} X$) we have
\[Y\vee\bigvee_{j=1}^k\Theta_X(Y_{i_j})=\Theta_X^{-1}\big(\Theta_X(Y)\cap\bigcap_{j=1}^k\Theta_X(Y_{i_j})\big)
=\Theta_X^{-1}(\beta X\backslash\lambda_{{\mathcal P}} X)=M^X_{{\mathcal P}}.\]

(2) {\em  implies} (1).  Let  $\{U_i\}_{i\in I}$  be an open cover of $\Theta_X(Y)\cap\lambda_{{\mathcal P}} X$
in $\beta X\backslash X$. Note that, since $X$  is locally-${\mathcal P}$ we have $X\subseteq\lambda_{{\mathcal P}} X$
(see Lemmas \ref{HGJHG} and \ref{JHG}), and since $X$ is non-${\mathcal P}$ we have $\beta X\backslash\lambda_{{\mathcal P}} X\neq\emptyset$
(see Lemma \ref{KJHG}, and  when ${\mathcal P}$ is pseudocompactness,
see Lemma \ref{JIUH}). For each $i\in I$ the set $(\beta X\backslash X)\backslash(U_i\cap\lambda_{{\mathcal P}} X)$ is compact (as it is closed in  $\beta X\backslash X$ and the latter is compact, as  $X$ is locally compact) and  it is non-empty (as it contains
$\beta X\backslash\lambda_{{\mathcal P}} X$). By Lemma  \ref{JHKFH} for each $i\in I$ we have
\[\Theta_X^{-1}\big((\beta X\backslash X)\backslash(U_i\cap\lambda_{{\mathcal P}} X)\big)=Y_i\]
for some $Y_i\in{\mathscr E}_{{\mathcal P}}(X)$. Now
\begin{eqnarray*}
\Theta_X(Y)\cap\bigcap_{i\in I}\Theta_X(Y_i)&=&\Theta_X(Y)\cap\bigcap_{i\in I}\big((\beta X\backslash X)
\backslash(U_i\cap\lambda_{{\mathcal P}} X)\big)\\&=&\Theta_X(Y)\backslash\bigcup_{i\in I}(U_i\cap\lambda_{{\mathcal P}} X)\\&\subseteq&\Theta_X(Y)
\backslash\big(\Theta_X(Y)\cap\lambda_{{\mathcal P}} X\big)\\&=&\Theta_X(Y)\cap(\beta X\backslash\lambda_{{\mathcal P}} X)=\beta X\backslash\lambda_{{\mathcal P}} X
\end{eqnarray*}
and thus,  by Lemma \ref{KJGT} (note that by Lemma  \ref{JHKFH} each
element of $\Theta_X({\mathscr E}_{{\mathcal P}}(X))$ contains $\beta X\backslash\lambda_{{\mathcal P}} X$, therefore,  equality holds in above) we have
\[Y\vee\bigvee_{i\in I}Y_i=\Theta_X^{-1}\Big(\Theta_X(Y)\cap\bigcap_{i\in I}\Theta_X(Y_i)\Big)=\Theta_X^{-1}(\beta X\backslash\lambda_{{\mathcal P}} X)
=M^X_{{\mathcal P}}.\]
Using our assumption, it now  follows that
\[Y\vee \bigvee_{j=1}^k\Theta_X(Y_{i_j})=M^X_{{\mathcal P}}\]
for some $i_1,\ldots,i_k\in I$. Again, by Lemma \ref{KJGT} we have
\begin{eqnarray*}
\lambda_{{\mathcal P}} X\cap\Theta_X(Y)\cap\bigcap_{j=1}^k\Theta_X(Y_{i_j})&=&\lambda_{{\mathcal P}} X\cap\Theta_X^{-1}\Big(Y\vee \bigvee_{j=1}^kY_{i_j}\Big)
\\&=&\lambda_{{\mathcal P}} X\cap\Theta_X^{-1}(M^X_{{\mathcal P}})=\lambda_{{\mathcal P}} X\cap(\beta X\backslash\lambda_{{\mathcal P}} X)=\emptyset
\end{eqnarray*}
and thus, since
\begin{eqnarray*}
\big(\lambda_{{\mathcal P}} X\cap\Theta_X(Y)\big)\backslash\bigcup_{j=1}^kU_{i_j}&\subseteq&\lambda_{{\mathcal P}} X\cap\Theta_X(Y)\cap
\Big((\beta X\backslash X)\backslash\bigcup_{j=1}^k(\lambda_{{\mathcal P}} X\cap U_{i_j})\Big)\\&=&\lambda_{{\mathcal P}} X\cap\Theta_X(Y)\cap\bigcap_{j=1}^k
\big((\beta X\backslash X)\backslash(\lambda_{{\mathcal P}} X\cap U_{i_j})\big)\\&=&\lambda_{{\mathcal P}} X\cap\Theta_X(Y)\cap\bigcap_{j=1}^k\Theta_X(Y_{i_j})
\end{eqnarray*}
comparing with above, it follows that
\[\big(\lambda_{{\mathcal P}} X\cap\Theta_X(Y)\big)\backslash\bigcup_{j=1}^kU_{i_j}=\emptyset.\]
Therefore
\[\lambda_{{\mathcal P}} X\cap\Theta_X(Y)\subseteq \bigcup_{j=1}^kU_{i_j}.\]
This shows the compactness of  $\lambda_{{\mathcal P}} X\cap\Theta_X(Y)$.
\end{proof}

\begin{theorem}\label{JOOGF}
Let $X$ and $Y$ be locally compact locally-${\mathcal P}$ non-${\mathcal P}$  spaces  where  ${\mathcal P}$ is either pseudocompactness or a
closed hereditary Mr\'{o}wka topological  property which is preserved under  finite closed sums of subspaces.
Then every order isomorphism from $({\mathscr E}_{{\mathcal P}}(X),\leq)$ onto $({\mathscr E}_{{\mathcal P}}(Y),\leq)$ restricts to
an  order-isomorphism from $({\mathscr E}^C_{{\mathcal P}}(X),\leq)$ onto $({\mathscr E}^C_{{\mathcal P}}(Y),\leq)$. Conversely,
every order-isomorphism from
$({\mathscr E}^C_{{\mathcal P}}(X),\leq)$ onto $({\mathscr E}^C_{{\mathcal P}}(Y),\leq)$ extends to an order-isomorphism from
$({\mathscr E}_{{\mathcal P}}(X),\leq)$ onto $({\mathscr E}_{{\mathcal P}}(Y),\leq)$.
\end{theorem}

\begin{proof}
Let
\[\Delta:\big({\mathscr E}_{{\mathcal P}}(X),\leq\big)\rightarrow\big({\mathscr E}_{{\mathcal P}}(Y),\leq\big)\]
be an  order-isomorphism. Let $T\in{\mathscr E}^C_{{\mathcal P}}(X)$. By Lemmas  \ref{HDJYYH} and  \ref{JHKFH} we have
$\Theta_X(T)\in {\mathscr Z}(\beta X\backslash X)$ and $\beta X\backslash\lambda_{{\mathcal P}} X\subseteq \Theta_X(T)$. The set
$(\beta X\backslash X)\backslash \Theta_X(T)$ is a cozero-set in  $\beta X\backslash X$, and the latter is compact (as $X$ is locally compact), therefore, $(\beta X\backslash X)\backslash \Theta_X(T)$  is $\sigma$-compact. Let
\begin{equation}\label{HFSD}
(\beta X\backslash X)\backslash\Theta_X(T)=\bigcup_{n=1}^\infty K_n
\end{equation}
where $K_n$'s, for $n\in \mathbf{N}$, are compact subsets of
$\lambda_{{\mathcal P}} X\backslash X$ (as $\beta X\backslash\lambda_{{\mathcal P}} X\subseteq \Theta_X(T)$). Let
\begin{equation}\label{JKGDF}
T_n=\Theta_X^{-1}\big((\beta X\backslash\lambda_{{\mathcal P}} X)\cup K_n\big)
\end{equation}
for $n\in \mathbf{N}$. By Lemma \ref{JHKFH}, we have  $T_n\in{\mathscr E}_{{\mathcal P}}(X)$ for each $n\in \mathbf{N}$ (note that $\beta X\backslash\lambda_{{\mathcal P}} X\neq\emptyset$, as  $X$ is non-${\mathcal P}$, see Lemma \ref{KJHG}, and when  ${\mathcal P}$ is pseudocompactness,
see Lemma \ref{JIUH}). Since $\Theta_X(T_n)\cap\lambda_{{\mathcal P}} X=K_n$ is compact, by Lemma \ref{HGTIITD}, the set
$L_n=\Theta_Y(\Delta(T_n))\cap\lambda_{{\mathcal P}} Y$ also is compact for $n\in \mathbf{N}$.
Using Lemma \ref{KJGT}, and since
\[\Theta_X(T_n\vee T)=\Theta_X(T_n)\cap\Theta_X(T)=\beta X\backslash\lambda_{{\mathcal P}} X=\Theta_X(M^X_{{\mathcal P}})\]
we have $T_n\vee T=M^X_{{\mathcal P}}$, and thus
\begin{eqnarray*}
\Theta_Y\big(\Delta(T_n)\big)\cap\Theta_Y\big(\Delta(T)\big)&=&\Theta_Y\big(\Delta(T_n)\vee \Delta(T)\big)\\&=&\Theta_Y\big(\Delta(T_n\vee T)\big)\\&=&\Theta_Y\big(\Delta(M^X_{{\mathcal P}})\big)
=\Theta_Y(M^Y_{{\mathcal P}})=\beta Y\backslash\lambda_{{\mathcal P}} Y
\end{eqnarray*}
for each $n\in \mathbf{N}$. Therefore
\begin{equation}\label{JHGF}
L_n\cap \Theta_Y\big(\Delta(T)\big)=\emptyset\mbox{ for }n\in \mathbf{N}.
\end{equation}
We verify that
\begin{equation}\label{JHGHL}
(\beta Y\backslash Y)\backslash\Theta_Y\big(\Delta(T)\big)=\bigcup_{n=1}^\infty L_n
\end{equation}
from which it will then follow that $\Theta_Y(\Delta(T))\in {\mathscr Z}(\beta Y\backslash Y)$. To see this, note that $\beta Y\backslash Y$ is normal, as it is compact,  as $Y$ is locally compact. Let
$h_n:\beta Y\backslash Y\rightarrow\mathbf{I}$, for $n\in \mathbf{N}$, be continuous and such that $h_n|L_n\equiv 1$ and
$h_n|\Theta_Y(\Delta(T))\equiv 0$. Then $\Theta_Y(\Delta(T))=Z(h)$ where $h=\sum_{n=1}^\infty h_n/2^n$. But, to show (\ref{JHGHL}), by (\ref{JHGF}),
it suffices to verify that
\[(\beta Y\backslash Y)\backslash\Theta_Y\big(\Delta(T)\big)\subseteq\bigcup_{n=1}^\infty L_n.\]
Let $s\in (\beta Y\backslash Y)\backslash\Theta_Y(\Delta(T))$ and suppose to the contrary that $s\notin L_n$ for each $n\in \mathbf{N}$. Since $\Delta(T)\in{\mathscr E}_{{\mathcal P}}(Y)$, by  Lemma \ref{JHKFH} we have
$\beta Y\backslash\lambda_{{\mathcal P}} Y\subseteq \Theta_Y(\Delta(T))$, and thus $s\in \lambda_{{\mathcal P}} Y$. Let
\[S=\Theta_Y^{-1}\big((\beta Y\backslash\lambda_{{\mathcal P}} Y)\cup\{s\}\big).\]
Then by Lemma \ref{JHKFH} we have $S\in{\mathscr E}_{{\mathcal P}}(Y)$. Using Lemma \ref{KJGT}, and since
\begin{eqnarray*}
\Theta_Y\big(S\vee\Delta(T_n)\big)&=&\Theta_Y(S)\cap\Theta_Y\big(\Delta(T_n)\big)\\&=&\big((\beta Y\backslash\lambda_{{\mathcal P}} Y)\cup\{s\}\big)\cap\big((\beta Y\backslash\lambda_{{\mathcal P}} Y)\cup L_n\big)\\&=&(\beta Y\backslash\lambda_{{\mathcal P}} Y)\cup\big(\{s\}\cap\big((\beta Y\backslash\lambda_{{\mathcal P}} Y)\cup L_n\big)\big)\\&=&\beta Y\backslash\lambda_{{\mathcal P}} Y=\Theta_Y(M^Y_{{\mathcal P}})
\end{eqnarray*}
for $n\in \mathbf{N}$, we have $S\vee\Delta(T_n)=M^Y_{{\mathcal P}}$. Similarly, $S\vee\Delta(T)=M^Y_{{\mathcal P}}$. Therefore
\[\Delta^{-1}(S)\vee T_n=\Delta^{-1}(S)\vee\Delta^{-1}\big(\Delta(T_n)\big)=\Delta^{-1}\big(S\vee\Delta(T_n)\big)=\Delta^{-1}(M^Y_{{\mathcal P}})=M^X_{{\mathcal P}}\]
for $n\in \mathbf{N}$, and similarly,
$\Delta^{-1}(S)\vee T=M^X_{{\mathcal P}}$. But, now Lemma \ref{KJGT} implies that
\begin{equation}\label{KHF}
\Theta_X\big(\Delta^{-1}(S)\big)\cap\Theta_X(T_n)=\Theta_X\big(\Delta^{-1}(S)\vee T_n\big)=\Theta_X(M^X_{{\mathcal P}})=\beta X\backslash\lambda_{{\mathcal P}} X
\end{equation}
for $n\in \mathbf{N}$, and similarly,
\[\Theta_X\big(\Delta^{-1}(S)\big)\cap\Theta_X(T)=\beta X\backslash\lambda_{{\mathcal P}} X.\]
Note that by (\ref{JKGDF}) we have $\Theta_X(T_n)=(\beta X\backslash\lambda_{{\mathcal P}} X)\cup K_n$ for $n\in \mathbf{N}$ and thus using (\ref{KHF}) it follows that
\begin{eqnarray*}
\Theta_X\big(\Delta^{-1}(S)\big)\cap\bigcup_{n=1}^\infty K_n&\subseteq&\bigcup_{n=1}^\infty\big(\Theta_X\big(\Delta^{-1}(S)\big)\cap K_n\big)\\&\subseteq&\bigcup_{n=1}^\infty\big(\Theta_X\big(\Delta^{-1}(S)\big)\cap \Theta_X(T_n)\big)\subseteq\beta X\backslash\lambda_{{\mathcal P}} X
\end{eqnarray*}
and therefore
\begin{eqnarray*}
\Theta_X\big(\Delta^{-1}(S)\big)&=&\Theta_X\big(\Delta^{-1}(S)\big)\cap\big(\Theta_X(T)\cup\big((\beta X\backslash X)\backslash\Theta_X(T)\big)\big)\\&=&\big(\Theta_X\big(\Delta^{-1}(S)\big)\cap\Theta_X(T)\big)\cup\big(\Theta_X\big(\Delta^{-1}(S)\big)\cap\big((\beta X\backslash X)\backslash\Theta_X(T)\big)\big)
\\&=&(\beta X\backslash\lambda_{{\mathcal P}} X)\cup\Big(\Theta_X\big(\Delta^{-1}(S)\big)\cap\bigcup_{n=1}^\infty K_n\Big)=\beta X\backslash\lambda_{{\mathcal P}} X.
\end{eqnarray*}
But this is not possible (see (\ref{HFSD}) and (\ref{JKGDF})), as $\Delta^{-1}(S)\neq M^X_{{\mathcal P}}$,
as $S\neq M^Y_{{\mathcal P}}$. Thus $\Theta_Y(\Delta(T))\in {\mathscr Z}(\beta Y\backslash Y)$, which (since
$\Delta(T)\in{\mathscr E}_{{\mathcal P}}(Y)$)
by Lemma \ref{HDJYYH} implies that $\Delta(T)\in{\mathscr E}^C_{{\mathcal P}}(Y)$. Therefore
$\Delta({\mathscr E}^C_{{\mathcal P}}(X))\subseteq{\mathscr E}^C_{{\mathcal P}}(Y)$. A similar argument shows that
$\Delta^{-1}({\mathscr E}^C_{{\mathcal P}}(Y))\subseteq{\mathscr E}^C_{{\mathcal P}}(X)$. Thus $\Delta({\mathscr E}^C_{{\mathcal P}}(X))={\mathscr E}^C_{{\mathcal P}}(Y)$.

To show  the converse, let
\[\delta:\big({\mathscr E}^C_{{\mathcal P}}(X),\leq\big)\rightarrow\big({\mathscr E}^C_{{\mathcal P}}(Y),\leq\big)\]
be an order-isomorphism.  Let
\[f=\Theta_Y\delta\Theta_X^{-1}:\big(\Theta_X\big({\mathscr E}^C_{{\mathcal P}}(X)\big),\subseteq\big)\rightarrow\big(\Theta_Y\big({\mathscr E}^C_{{\mathcal P}}(Y)\big),\subseteq\big).\]
Then $f$ is an order-isomorphism. We extend $f$ to an order-isomorphism
\[F:\big(\Theta_X\big({\mathscr E}_{{\mathcal P}}(X)\big),\subseteq\big)\rightarrow\big(\Theta_Y\big({\mathscr E}_{{\mathcal P}}(Y)\big),\subseteq\big).\]
Let $C\in \Theta_X({\mathscr E}_{{\mathcal P}}(X))$. By Lemma \ref{JHKFH}, the set $C$ is a non-empty compact subset of
$\beta X\backslash X$ containing $\beta X\backslash\lambda_{{\mathcal P}} X$. Note that the set of zero-sets of any
Tychonoff space is a base for its closed sets. Then  $C=\bigcap_{i\in I}Z_i$, where
$Z_i\in {\mathscr Z}(\beta X\backslash X)$ for $i\in I$ and $I$ is an index set. Since  $\beta X\backslash\lambda_{{\mathcal P}} X\subseteq Z_i$,
for each $i\in I$, by Lemmas \ref{HDJYYH} and \ref{JHKFH} we have $Z_i\in \Theta_X({\mathscr E}^C_{{\mathcal P}}(X))$,
and thus  $f(Z_i)\in \Theta_Y({\mathscr E}^C_{{\mathcal P}}(Y))$. Define
\[F(C)= \bigcap_{i\in I}f(Z_i).\]
By Lemma  \ref{JHKFH} it follows that  $\beta Y\backslash\lambda_{{\mathcal P}} Y\subseteq f(Z_i)$ for  $i\in I$. Note that, since $Y$ is non-${\mathcal P}$ we have
$\beta Y\backslash\lambda_{{\mathcal P}} Y\neq\emptyset$ (see Lemma \ref{KJHG}, and  when ${\mathcal P}$ is pseudocompactness,
see Lemma \ref{JIUH}) for each $i\in I$. Also, $f(Z_i)$'s, for $i\in I$, are compact, as they are zero-sets in $\beta Y\backslash Y$ and the latter is compact, as  $Y$ is locally compact. We show that $F$ is well-defined, i.e.,
its definition is independent of the choices of $Z_i$'s.

Let $C\in \Theta_X({\mathscr E}_{{\mathcal P}}(X))$, and let
\[C=\bigcap_{i\in I}Z_i\mbox{ and }C=\bigcap_{j\in J}S_j\]
where $Z_i,S_j\in {\mathscr Z}(\beta X\backslash X)$ for $i\in I$ and  $j\in J$,  and $I$ and $J$ are index sets, be two representations for $C$. Let
$g_i:\beta X\backslash X\rightarrow\mathbf{I}$, for  $i\in I$, be continuous and such that $Z_i=Z(g_i)$. Fix some $i'\in I$ and some $n'\in \mathbf{N}$. Since
\[\bigcap_{j\in J}S_j=C\subseteq Z_{i'}=Z(g_{i'})\subseteq g_{i'}^{-1}\big([0,1/n')\big)\]
by compactness of $\beta X\backslash X$ (as $X$ is locally compact) and since the latter is open in $\beta X\backslash X$ we have
\[\bigcap_{i=1}^k S_{j_i}\subseteq g_{i'}^{-1}\big([0,1/n')\big)\subseteq g_{i'}^{-1}\big([0,1/n']\big)\]
for some $j_1,\ldots,j_k\in J$. Since
\[\beta X\backslash\lambda_{{\mathcal P}} X \subseteq C\subseteq Z_{i'}\subseteq g_{i'}^{-1}\big([0,1/n']\big)\in {\mathscr Z}(\beta X\backslash X)\]
by Lemmas \ref{HDJYYH} and \ref{JHKFH} we have $g_{i'}^{-1}([0,1/n'])\in\Theta_X({\mathscr E}^C_{{\mathcal P}}(X))$.
Since $S_j$'s for $j\in J$ (and their countable intersections) are in  $\Theta_X({\mathscr E}^C_{{\mathcal P}}(X))$ and $f$ is an  order-isomorphism, we have
\[\bigcap_{j\in J}f(S_j)\subseteq \bigcap_{i=1}^k f(S_{j_i})= f\Big(\bigcap_{i=1}^k S_{j_i}\Big)\subseteq f\big(g_{i'}^{-1}\big([0,1/n']\big)\big).\]
Now, since $n'\in \mathbf{N}$ is arbitrary, we have
\[\bigcap_{j\in J}f(S_j)\subseteq \bigcap_{n=1}^\infty f\big(g_{i'}^{-1}\big([0,1/n]\big)\big)=f\Big(\bigcap_{n=1}^\infty  g_{i'}^{-1}\big([0,1/n]\big)\Big)=f\big(Z( g_{i'})\big)=f(Z_{i'})\]
and since  $i'\in I$  is arbitrary, it follows that
\[\bigcap_{j\in J}f(S_j)\subseteq \bigcap_{i\in I}f(Z_i).\]
The reverse inclusion can be proved analogously. This shows that  $F$ is well-defined.

Now, we show that $F$ is an order-homomorphism. Let $C,D\in\Theta_X({\mathscr E}_{{\mathcal P}}(X))$ be such that $C\subseteq D$.
Let
\[C=\bigcap_{i\in I}Z_i\mbox{ and }D=\bigcap_{j\in J}S_j\]
where $Z_i,S_j\in {\mathscr Z}(\beta X\backslash X)$ for $i\in I$ and  $j\in J$,  and $I$ and $J$ are index sets. Then
\[C=C\cap D=\bigcap_{i\in I}Z_i\cap\bigcap_{j\in J}S_j\]
and therefore
\[F(C)=\bigcap_{i\in I}f(Z_i)\cap\bigcap_{j\in J}f(S_j)\subseteq\bigcap_{j\in J}f(S_j)=F(D).\]

Similarly, the order-isomorphism
\[f^{-1}:\big(\Theta_Y\big({\mathscr E}^C_{{\mathcal P}}(Y)\big),\subseteq\big)\rightarrow\big(\Theta_X\big({\mathscr E}^C_{{\mathcal P}}(X)\big),\subseteq\big).\]
induces an order-homomorphism
\[G:\big(\Theta_Y\big({\mathscr E}_{{\mathcal P}}(Y)\big),\subseteq\big)\rightarrow\big(\Theta_X\big({\mathscr E}_{{\mathcal P}}(X)\big),\subseteq\big)\]
defined by
\[G(D)= \bigcap_{j\in J}f^{-1}(S_j)\]
for $D\in \Theta_Y({\mathscr E}_{{\mathcal P}}(Y))$, whenever  $D=\bigcap_{j\in J}S_j$, where
$S_j\in {\mathscr Z}(\beta Y\backslash Y)$ for $j\in J$  and $J$ is an index set. Then for every  $C\in \Theta_X({\mathscr E}_{{\mathcal P}}(X))$,
that $C=\bigcap_{i\in I}Z_i$, where
$Z_i\in {\mathscr Z}(\beta X\backslash X)$ for $i\in I$ and  $I$ is an index set (since $f(Z_i)\in  {\mathscr Z}(\beta Y\backslash Y)$), we have
\[G\big(F(C)\big)=G\Big(\bigcap_{i\in I}f(Z_i)\Big)=\bigcap_{i\in I}f^{-1}\big(f(Z_i)\big)=\bigcap_{i\in I}Z_i=C.\]
Similarly, $F(G(D))=D$ for every $D\in \Theta_Y({\mathscr E}_{{\mathcal P}}(Y))$, which shows that $G=F^{-1}$. This shows that  $F$ is an  order-isomorphism.

By the way we have defined $F$, for every $C\in \Theta_X({\mathscr E}^C_{{\mathcal P}}(X))$ (since $C\in {\mathscr Z}(\beta X\backslash X)$, see
Lemma  \ref{HDJYYH}) we have $F(C)=f(C)$, i.e., $F$ extends $f$. If we now define
\[\Delta=\Theta_Y^{-1}F\Theta_X:\big({\mathscr E}_{{\mathcal P}}(X),\leq\big)\rightarrow\big({\mathscr E}_{{\mathcal P}}(Y),\leq\big)\]
then $\Delta$ is an  order-isomorphism, and $\Delta|{\mathscr E}^C_{{\mathcal P}}(X)=\delta$.
\end{proof}

\begin{theorem}\label{HGFFD}
Let $X$ and $Y$ be locally compact locally-${\mathcal P}$ non-${\mathcal P}$  spaces  where  ${\mathcal P}$ is either pseudocompactness or a
closed hereditary Mr\'{o}wka topological property which is preserved under  finite closed sums of subspaces.
Then every  order-isomorphism from
$({\mathscr E}^C_{{\mathcal P}}(X),\leq)$ onto $({\mathscr E}^C_{{\mathcal P}}(Y),\leq)$ restricts to an order-isomorphism
from $({\mathscr E}^K_{{\mathcal P}}(X),\leq)$ onto $({\mathscr E}^K_{{\mathcal P}}(Y),\leq)$.
Conversely, if $X$ and $Y$ are moreover  strongly zero-dimensional, then every order-isomorphism
from $({\mathscr E}^K_{{\mathcal P}}(X),\leq)$ onto $({\mathscr E}^K_{{\mathcal P}}(Y),\leq)$ extends to an order-isomorphism
from $({\mathscr E}^C_{{\mathcal P}}(X),\leq)$ onto $({\mathscr E}^C_{{\mathcal P}}(Y),\leq)$.
\end{theorem}

\begin{proof}
Let
\[\Delta:\big({\mathscr E}^C_{{\mathcal P}}(X),\leq\big)\rightarrow\big({\mathscr E}^C_{{\mathcal P}}(Y),\leq\big)\]
be an  order-isomorphism. Let $T\in {\mathscr E}^K_{{\mathcal P}}(X)$. Then $T\in {\mathscr E}^C_{{\mathcal P}}(X)$
and thus $\Delta(T)\in {\mathscr E}^C_{{\mathcal P}}(X)$. By Lemmas \ref{HDJYYH} and \ref{JHKFH}  we have  $\Theta_Y(\Delta(T))\in {\mathscr Z}(\beta Y\backslash Y)$ and
$\beta Y\backslash\lambda_{{\mathcal P}} Y\subseteq \Theta_Y(\Delta(T))$.
Let $h:\beta Y\backslash Y\rightarrow\mathbf{I}$ be continuous and such that $Z(h)=\Theta_Y(\Delta(T))$. For each $n\in \mathbf{N}$, since
$h^{-1}([0,1/n])\in  {\mathscr Z}(\beta Y\backslash Y)$ and
\[\beta Y\backslash\lambda_{{\mathcal P}} Y\subseteq \Theta_Y\big(\Delta(T)\big)=Z(h)\subseteq h^{-1}\big([0,1/n]\big)\]
by Lemmas \ref{HDJYYH} and \ref{JHKFH}, we have   $h^{-1}([0,1/n])\in\Theta_Y({\mathscr E}^C_{{\mathcal P}}(Y))$. Let
\[S_n=\Theta_Y^{-1}\big( h^{-1}\big([0,1/n]\big)\big)\in {\mathscr E}^C_{{\mathcal P}}(Y)\]
for $n\in \mathbf{N}$. Using Lemma \ref{KJGT}, it follows that
\begin{eqnarray*}
\bigcap_{n=1}^\infty\Theta_X\big(\Delta^{-1}(S_n)\big)&=&\Theta_X\Big(\bigvee_{n=1}^\infty\Delta^{-1}(S_n)\Big)\\&=&
\Theta_X\Big(\Delta^{-1}\Big(\bigvee_{n=1}^\infty S_n\Big)\Big)\\&=&\Theta_X\Big(\Delta^{-1}\Big(\Theta_Y^{-1}\Big(\bigcap_{n=1}^\infty \Theta_Y(S_n)\Big)\Big)\Big)
\\&=&\Theta_X\Big(\Delta^{-1}\Big(\Theta_Y^{-1}\Big(\bigcap_{n=1}^\infty  h^{-1}\big([0,1/n]\big)\Big)\Big)\Big)
\\&=&\Theta_X\big(\Delta^{-1}\big(\Theta_Y^{-1}\big(Z(h)\big)\big)\big)
\\&=&\Theta_X\big(\Delta^{-1}\big(\Theta_Y^{-1}\big(\Theta_Y\big(\Delta(T)\big)\big)\big)\big)=\Theta_X(T).
\end{eqnarray*}
Therefore
\[\big((\beta X\backslash X)\backslash\Theta_X(T)\big)\cap\bigcap_{n=1}^\infty\Theta_X\big(\Delta^{-1}(S_n)\big)=\emptyset\]
and, thus, since $(\beta X\backslash X)\backslash\Theta_X(T)$ is compact (as $\Theta_X(T)$ is open in $\beta X\backslash X$ and the latter is compact, as $X$ is locally compact) we have
\begin{equation}\label{HGTTD}
\big((\beta X\backslash X)\backslash\Theta_X(T)\big)\cap\bigcap_{i=1}^k\Theta_X\big(\Delta^{-1}(S_{n_i})\big)=\emptyset
\end{equation}
for some $n_1,\ldots,n_k\in  \mathbf{N}$. Arguing as above, it follows that
\begin{eqnarray*}
\bigcap_{i=1}^k\Theta_X\big(\Delta^{-1}(S_{n_i})\big)&=&\Theta_X\Big(\Delta^{-1}\Big(\Theta_Y^{-1}\Big(\bigcap_{i=1}^k  h^{-1}\big([0,1/n_i]\big)\Big)\Big)\Big)
\\&=&\Theta_X\big(\Delta^{-1}\big(\Theta_Y^{-1}\big(h^{-1}\big([0,1/t]\big)\big)\big)\big)
\end{eqnarray*}
where $t=\max\{n_1,\ldots,n_k\}$. Therefore, by (\ref{HGTTD}) we have
\[\Theta_X\big(\Delta^{-1}\big(\Theta_Y^{-1}\big(h^{-1}\big([0,1/t]\big)\big)\big)\big)\subseteq\Theta_X(T).\]
Thus
\[h^{-1}\big([0,1/t]\big)\subseteq\Theta_Y\big(\Delta(T)\big)=Z(h)\subseteq h^{-1}\big([0,1/t)\big)\]
and therefore, the set
\[\Theta_Y\big(\Delta(T)\big)= h^{-1}\big([0,1/t)\big)\]
is clopen in $\beta Y\backslash Y$. By Lemma \ref{FSYH}  it follows that $\Delta(T)\in {\mathscr E}^K_{{\mathcal P}}(Y)$, i.e.,
$\Delta({\mathscr E}^K_{{\mathcal P}}(X))\subseteq {\mathscr E}^K_{{\mathcal P}}(Y)$.
A similar argument shows that $\Delta^{-1}({\mathscr E}^K_{{\mathcal P}}(Y))\subseteq {\mathscr E}^K_{{\mathcal P}}(X)$. Thus
$\Delta({\mathscr E}^K_{{\mathcal P}}(X))={\mathscr E}^K_{{\mathcal P}}(Y)$.

To show the converse, let
\[\delta:\big({\mathscr E}^K_{{\mathcal P}}(X),\leq\big)\rightarrow\big({\mathscr E}^K_{{\mathcal P}}(Y),\leq\big)\]
be an order-isomorphism.  Let
\[f=\Theta_Y\delta\Theta_X^{-1}:\big(\Theta_X\big({\mathscr E}^K_{{\mathcal P}}(X)\big),\subseteq\big)\rightarrow\big(\Theta_Y\big({\mathscr E}^K_{{\mathcal P}}(Y)\big),\subseteq\big).\]
Then $f$ is an order-isomorphism. We extend $f$ to an order-isomorphism
\[F:\big(\Theta_X\big({\mathscr E}^C_{{\mathcal P}}(X)\big),\subseteq\big)\rightarrow\big(\Theta_Y\big({\mathscr E}^C_{{\mathcal P}}(Y)\big),\subseteq\big).\]
Let $Z\in \Theta_X({\mathscr E}^C_{{\mathcal P}}(X))$. By Lemmas \ref{HDJYYH} and \ref{JHKFH} we have  $Z\in {\mathscr Z}(\beta X\backslash X)$, $Z$ is non-empty
and  $\beta X\backslash\lambda_{{\mathcal P}} X\subseteq Z$. Let  $Z'\in {\mathscr Z}(\beta X)$ be such that $Z'\cap(\beta X\backslash X)=Z$ (such a $Z'$ exists, as $\beta X\backslash X$ is closed in the normal space $\beta X$, as $X$ is locally compact, and therefore, by the Tietze-Urysohn Theorem,
every continuous function of
$\beta X\backslash X$ into $\mathbf{I}$ is continuously extendible over $\beta X$). Note that  in a strongly zero-dimensional space every zero-set is an intersection of countably many  clopen subsets (see Theorem 4.7 (j) of \cite{PW}). Since $X$  is strongly zero-dimensional, $\beta X$ also is
strongly zero-dimensional. Thus $Z'=\bigcap_{n=1}^\infty U_n'$, where $U_n'$'s are clopen subsets of $\beta X$ for $n\in \mathbf{N}$. Therefore, if we let  $U_n=U_n'\cap(\beta X\backslash X)$ for $n\in \mathbf{N}$, then  $Z=\bigcap_{n=1}^\infty U_n$. Assume such a representation for $Z$, i.e., assume that
$Z=\bigcap_{n=1}^\infty U_n$, where $U_n$'s are clopen subsets of $\beta X\backslash X$ for $n\in \mathbf{N}$, and define
\[F(Z)=\bigcap_{n=1}^\infty f(Z_n).\]
Note that, for each $n\in \mathbf{N}$,  since $U_n$ is clopen in $\beta X\backslash X$ and
$\beta X\backslash\lambda_{{\mathcal P}} X\subseteq Z\subseteq U_n$, we have $U_n\in \Theta_X({\mathscr E}^K_{{\mathcal P}}(X))$, and thus
$f(U_n)\in \Theta_Y({\mathscr E}^K_{{\mathcal P}}(Y))$.  Thus, by  Lemma \ref{FSYH}, for each $n\in \mathbf{N}$ the set $f(U_n)$ is  clopen in $\beta Y\backslash Y$
and $\beta Y\backslash\lambda_{{\mathcal P}} Y\subseteq f(U_n)$. Therefore
$\bigcap_{n=1}^\infty f(U_n)$ is a zero-set in $\beta Y\backslash Y$  containing
$\beta Y\backslash\lambda_{{\mathcal P}} Y$ (note that the latter set is non-empty, as $Y$ is  non-${\mathcal P}$; see Lemma \ref{KJHG})
which shows that $\bigcap_{n=1}^\infty f(U_n)\in \Theta_Y({\mathscr E}^C_{{\mathcal P}}(Y))$.
To show that $F$ is well-defined, we need to verify that its definition is independent of the choices of $U_n$'s.

Let $Z\in \Theta_X({\mathscr E}^C_{{\mathcal P}}(X))$, and let
\[Z=\bigcap_{n=1}^\infty U_n\mbox{ and }Z=\bigcap_{n=1}^\infty V_n\]
where $U_n$'s and $V_n$'s, for $n\in \mathbf{N}$, are clopen in $\beta X\backslash X$, be two representations for $Z$.  Fix some  $n'\in \mathbf{N}$. Since
\[\bigcap_{n=1}^\infty U_n=Z\subseteq V_{n'}\]
by compactness of $\beta X\backslash X$ (as $X$ is locally compact) and since $V_{n'}$ is open in $\beta X\backslash X$ we have
\[\bigcap_{i=1}^kU_{n_i}\subseteq V_{n'}\]
for some $n_1,\ldots,n_k\in \mathbf{N}$. Now, since $U_n$'s for $n\in \mathbf{N}$ (and their finite intersections) are
in  $\Theta_X({\mathscr E}^K_{{\mathcal P}}(X))$ and $f$ is an  order-isomorphism, we have
\[\bigcap_{n=1}^\infty f(U_n)\subseteq\bigcap_{i=1}^kf(U_{n_i})= f\Big(\bigcap_{i=1}^kU_{n_i}\Big)\subseteq f( V_{n'}).\]
Since $n'\in \mathbf{N}$ is arbitrary, it follows that
\[\bigcap_{n=1}^\infty f(U_n)\subseteq \bigcap_{n=1}^\infty f(V_n).\]
The reverse inclusion can be proved analogously. This shows that  $F$ is well-defined.

Now, we show that $F$ is an order-homomorphism. Let $Z,S\in\Theta_X({\mathscr E}^C_{{\mathcal P}}(X))$ be such that $Z\subseteq S$.
Let
\[Z=\bigcap_{n=1}^\infty U_n\mbox{ and }S=\bigcap_{n=1}^\infty V_n\]
where
$U_n$'s and $V_n$'s are clopen subsets of $\beta X\backslash X$ for $n\in \mathbf{N}$.
Then
\[Z=Z\cap S=\bigcap_{n=1}^\infty U_n\cap\bigcap_{n=1}^\infty V_n\]
and thus
\[F(Z)=\bigcap_{n=1}^\infty f(U_n)\cap\bigcap_{n=1}^\infty f(V_n)\subseteq\bigcap_{n=1}^\infty f(V_n)=F(S).\]

Similarly, the order-isomorphism
\[f^{-1}:\big(\Theta_Y\big({\mathscr E}^K_{{\mathcal P}}(Y)\big),\subseteq\big)\rightarrow\big(\Theta_X\big({\mathscr E}^K_{{\mathcal P}}(X)\big),\subseteq\big).\]
induces an order-homomorphism
\[G:\big(\Theta_Y\big({\mathscr E}^C_{{\mathcal P}}(Y)\big),\subseteq\big)\rightarrow\big(\Theta_X\big({\mathscr E}^C_{{\mathcal P}}(X)\big),\subseteq\big)\]
defined by
\[G(S)=\bigcap_{n=1}^\infty f^{-1}(V_n)\]
for $S\in \Theta_Y({\mathscr E}^C_{{\mathcal P}}(Y))$, whenever  $S=\bigcap_{n=1}^\infty V_n$, where $V_n$'s,
for $n\in \mathbf{N}$, are clopen in $\beta Y\backslash Y$.
Now, for $Z\in \Theta_X({\mathscr E}^C_{{\mathcal P}}(X))$,
if $Z=\bigcap_{n=1}^\infty U_n$, where  $U_n$'s  are clopen in $\beta X\backslash X$ for $n\in \mathbf{N}$
(since $f(U_n)$'s  are clopen in  $\beta Y\backslash Y$  for $n\in \mathbf{N}$),  we have
\[G\big(F(Z)\big)=G\Big(\bigcap_{n=1}^\infty f(U_n)\Big)=\bigcap_{n=1}^\infty f^{-1}\big(f(U_n)\big)=\bigcap_{n=1}^\infty U_n=Z.\]
Similarly, $F(G(S))=S$ for every $S\in\Theta_Y({\mathscr E}^C_{{\mathcal P}}(Y))$, which shows that $G=F^{-1}$. Therefore, $F$ is an  order-isomorphism.

By the way we have defined $F$, for every $U\in \Theta_X({\mathscr E}^K_{{\mathcal P}}(X))$ (since by
Lemma \ref{FSYH} the set $U$ is clopen in $\beta X\backslash X$) we have $F(U)=f(U)$, i.e., $F$ extends $f$. Thus, if we let
\[\Delta=\Theta_Y^{-1}F\Theta_X:\big({\mathscr E}^C_{{\mathcal P}}(X),\leq\big)\rightarrow\big({\mathscr E}^C_{{\mathcal P}}(Y),\leq\big)\]
then $\Delta$ is an  order-isomorphism and $\Delta|{\mathscr E}^K_{{\mathcal P}}(X)=\delta$.
\end{proof}

\section{Relations between the order-structure of sets of one-point ${\mathcal P}$-extensions of a space  and the topologies  of subspaces of its outgrowth}

We begin this section with the following straightforward result from \cite{Ko2}. This   also provides motivation for what follows next.

\begin{theorem}\label{FSFH}
Let $X$ and $Y$ be locally compact spaces. The following are equivalent:
\begin{itemize}
\item[\rm(1)] $\beta X\backslash X$ and $\beta Y\backslash Y$ are homeomorphic.
\item[\rm(2)] $({\mathscr E}(X),\leq)$ and $({\mathscr E}(Y),\leq)$ are order-isomorphic.
\item[\rm(3)] $({\mathscr E}^C(X),\leq)$ and $({\mathscr E}^C(Y),\leq)$ are order-isomorphic.
\item[\rm(4)] $({\mathscr E}^K(X),\leq)$ and $({\mathscr E}^K(Y),\leq)$ are order-isomorphic, provided that $X$ and $Y$ are moreover strongly zero-dimensional.
\end{itemize}
\end{theorem}

\begin{proof}
This follows from the fact that in  a compact (compact and zero-dimensional, respectively) space the order-structure (partially ordered by $\subseteq$) of either the set of its closed subsets or the set of its zero-sets (the set of its  clopen subsets, respectively)  determines its topology  (see Theorem 11.1 of \cite{B} and Section 3.2 of \cite{PW}).
\end{proof}

\begin{notation}\label{HJFJB}
{\em For a space $X$ denote by ${\mathscr C}(X)$ the set of all closed subsets of $X$.}
\end{notation}

Part of the next theorem generalizes Theorem 3.21 and Corollary 4.11 of \cite{Ko2}.

\begin{theorem}\label{FHFH}
Let $X$ and $Y$ be locally compact locally-${\mathcal P}$ non-${\mathcal P}$  spaces,  where  ${\mathcal P}$ is either pseudocompactness, or a
closed hereditary Mr\'{o}wka topological property, which is preserved under  finite closed sums of subspaces.
The following are equivalent:
\begin{itemize}
\item[\rm(1)] $\lambda_{{\mathcal P}} X\backslash X$ and $\lambda_{{\mathcal P}} Y\backslash Y$ are homeomorphic.
\item[\rm(2)] $({\mathscr E}_{{\mathcal P}}(X),\leq)$ and $({\mathscr E}_{{\mathcal P}}(Y),\leq)$ are order-isomorphic.
\item[\rm(3)] $({\mathscr E}^C_{{\mathcal P}}(X),\leq)$ and $({\mathscr E}^C_{{\mathcal P}}(Y),\leq)$ are order-isomorphic.
\item[\rm(4)] $({\mathscr E}^K_{{\mathcal P}}(X),\leq)$ and $({\mathscr E}^K_{{\mathcal P}}(Y),\leq)$ are order-isomorphic, provided that $X$ and $Y$ are moreover strongly zero-dimensional.
\end{itemize}
\end{theorem}

\begin{proof}
We only prove the equivalence of (1) and (2). The equivalence of the remaining conditions follows from Theorems \ref{JOOGF} and \ref{HGFFD}.

(1) {\em  implies} (2).  Let $f:\lambda_{{\mathcal P}} X\backslash X\rightarrow\lambda_{{\mathcal P}} Y\backslash Y$ be a
homeomorphism. Let $S\in {\mathscr E}_{{\mathcal P}}(X)$. Since $\Theta_X(S)\cap(\lambda_{{\mathcal P}} X\backslash X)$ is closed in
$\lambda_{{\mathcal P}} X\backslash X$, the set $f(\Theta_X(S)\cap(\lambda_{{\mathcal P}} X\backslash X))$ is closed in  $\lambda_{{\mathcal P}} Y\backslash Y$, and thus (since $Y$ is locally compact)
\[(\beta Y\backslash\lambda_{{\mathcal P}} Y)\cup f\big(\Theta_X(S)\cap(\lambda_{{\mathcal P}} X\backslash X)\big)\]
is compact. To see this, note that if $A$ is a closed subset of $\beta Y\backslash Y$ such that
\[f\big(\Theta_X(S)\cap(\lambda_{{\mathcal P}} X\backslash X)\big)=A\cap(\lambda_{{\mathcal P}} Y\backslash Y)\]
then
\[(\beta Y\backslash\lambda_{{\mathcal P}} Y)\cup f\big(\Theta_X(S)\cap(\lambda_{{\mathcal P}} X\backslash X)\big)=A\cup(\beta Y\backslash\lambda_{{\mathcal P}} Y).\]
Define a function
\[F:\big({\mathscr E}_{{\mathcal P}}(X),\leq\big)\rightarrow\big({\mathscr E}_{{\mathcal P}}(Y),\leq\big)\]
by
\[F(S)=\Theta_Y^{-1}\big((\beta Y\backslash\lambda_{{\mathcal P}} Y)\cup f\big(\Theta_X(S)\cap(\lambda_{{\mathcal P}} X\backslash X)\big)\big)\]
for $S\in  {\mathscr E}_{{\mathcal P}}(X)$. We show that $F$ is an  order-isomorphism. The fact that $F$ is  an order-homomorphism is straightforward.
Let
\[G:\big({\mathscr E}_{{\mathcal P}}(Y),\leq\big)\rightarrow\big({\mathscr E}_{{\mathcal P}}(X),\leq\big)\]
be defined by
\[G(T)=\Theta_X^{-1}\big((\beta X\backslash\lambda_{{\mathcal P}} X)\cup f^{-1}\big(\Theta_Y(T)\cap(\lambda_{{\mathcal P}} Y\backslash Y)\big)\big)\]
for $T\in  {\mathscr E}_{{\mathcal P}}(Y)$. Note that for $S\in  {\mathscr E}_{{\mathcal P}}(X)$ we have
\begin{eqnarray*}
&&(\lambda_{{\mathcal P}} Y\backslash Y)\cap\Theta_Y\big(F(S)\big)\\&=&(\lambda_{{\mathcal P}} Y\backslash Y)\cap\Theta_Y\big(\Theta_Y^{-1}\big((\beta Y\backslash\lambda_{{\mathcal P}} Y)\cup f\big(\Theta_X(S)\cap(\lambda_{{\mathcal P}} X\backslash X)\big)\big)\big)\\&=&(\lambda_{{\mathcal P}} Y\backslash Y)\cap\big((\beta Y\backslash\lambda_{{\mathcal P}} Y)\cup f\big(\Theta_X(S)\cap(\lambda_{{\mathcal P}} X\backslash X)\big)\big)=f\big(\Theta_X(S)\cap(\lambda_{{\mathcal P}} X\backslash X)\big)
\end{eqnarray*}
and therefore
\begin{eqnarray*}
G\big(F(S)\big)&=&\Theta_X^{-1}\big((\beta X\backslash\lambda_{{\mathcal P}} X)\cup f^{-1}\big(\Theta_Y\big(F(S)\big)\cap(\lambda_{{\mathcal P}} Y\backslash Y)\big)\big)\\&=&\Theta_X^{-1}\big((\beta X\backslash\lambda_{{\mathcal P}} X)\cup f^{-1}\big(f\big(\Theta_X(S)\cap(\lambda_{{\mathcal P}} X\backslash X)\big)\big)\big)\\&=&\Theta_X^{-1}\big((\beta X\backslash\lambda_{{\mathcal P}} X)\cup \big(\Theta_X(S)\cap(\lambda_{{\mathcal P}} X\backslash X)\big)\big)=\Theta_X^{-1}\big(\Theta_X(S)\big)=S.
\end{eqnarray*}
Similarly, $F(G(T))=T$ for every  $T\in  {\mathscr E}_{{\mathcal P}}(Y)$. Thus $F^{-1}=G$.   Since $G$  is an  order-homomorphism, it follows that
$F$ is an  order-isomorphism.

(2) {\em  implies} (1). Let
\[F:\big({\mathscr E}_{{\mathcal P}}(X),\leq\big)\rightarrow\big({\mathscr E}_{{\mathcal P}}(Y),\leq\big)\]
be an order-isomorphism. Let $A\in {\mathscr C}(\lambda_{{\mathcal P}} X\backslash X)$. Let $A'$ be a closed subset of $\beta X\backslash X$ such that
$A=A'\cap(\lambda_{{\mathcal P}} X\backslash X)$. Note that $A'$ is compact, as it is closed in  $\beta X\backslash X$ and the latter is compact, as $X$ is locally compact.  Thus
(and since $\lambda_{{\mathcal P}} X$ is open in $\beta X$) the set
\[A\cup(\beta X\backslash\lambda_{{\mathcal P}} X)=A'\cup(\beta X\backslash\lambda_{{\mathcal P}} X)\]
is compact and non-empty (as $\beta X\backslash\lambda_{{\mathcal P}} X\neq\emptyset$, as $X$ is non-${\mathcal P}$, see  Lemma \ref{JIUH}  when  ${\mathcal P}$ is pseudocompactness, and see Lemma \ref{KJHG} in the other case). Define a function
\[f:\big({\mathscr C}(\lambda_{{\mathcal P}} X\backslash X),\subseteq\big)\rightarrow\big({\mathscr C}(\lambda_{{\mathcal P}} Y\backslash Y),\subseteq\big)\]
by
\[f(A)=(\lambda_{{\mathcal P}} Y\backslash Y)\cap\Theta_Y F\Theta_X^{-1}\big(A\cup(\beta X\backslash\lambda_{{\mathcal P}} X)\big)\]
for $A\in {\mathscr C}(\lambda_{{\mathcal P}} X\backslash X)$. Then $f$ is well-defined. We show that $f$ is an order-isomorphism. Since the order-structure of the set of  closed subsets of any Hausdorff space
determines its topology (see Theorem 11.1 of \cite{B}), this  will then prove (2). The fact that $f$ is an order-homomorphism is straightforward. Let the function
\[g:\big({\mathscr C}(\lambda_{{\mathcal P}} Y\backslash Y),\subseteq\big)\rightarrow\big({\mathscr C}(\lambda_{{\mathcal P}} X\backslash X),\subseteq\big)\]
be defined by
\[g(B)=(\lambda_{{\mathcal P}} X\backslash X)\cap\Theta_X F^{-1}\Theta_Y^{-1}\big(B\cup(\beta Y\backslash\lambda_{{\mathcal P}} Y)\big)\]
for $B\in {\mathscr C}(\lambda_{{\mathcal P}} Y\backslash Y)$. Note that for $A\in {\mathscr C}(\lambda_{{\mathcal P}} X\backslash X)$ we have
\begin{eqnarray*}
f(A)\cup(\beta Y\backslash\lambda_{{\mathcal P}} Y)&=&\big((\lambda_{{\mathcal P}} Y\backslash Y)\cap\Theta_Y F\Theta_X^{-1}\big(A\cup(\beta X\backslash\lambda_{{\mathcal P}} X)\big)\big)\cup(\beta Y\backslash\lambda_{{\mathcal P}} Y)\\&=&\Theta_Y F\Theta_X^{-1}\big(A\cup(\beta X\backslash\lambda_{{\mathcal P}} X)\big)
\end{eqnarray*}
and therefore
\begin{eqnarray*}
g\big(f(A)\big)&=&(\lambda_{{\mathcal P}} X\backslash X)\cap\Theta_X F^{-1}\Theta_Y^{-1}\big(f(A)\cup(\beta Y\backslash\lambda_{{\mathcal P}} Y)\big)\\&=&(\lambda_{{\mathcal P}} X\backslash X)\cap\Theta_X F^{-1}\Theta_Y^{-1}\big(\Theta_Y F\Theta_X^{-1}\big(A\cup(\beta X\backslash\lambda_{{\mathcal P}} X)\big)\big)\\&=&(\lambda_{{\mathcal P}} X\backslash X)\cap\big(A\cup(\beta X\backslash\lambda_{{\mathcal P}} X)\big)=A.
\end{eqnarray*}
Thus $gf=\mbox{id}_{{\mathscr C}(\lambda_{{\mathcal P}} X\backslash X)}$. Similarly,  $fg=\mbox{id}_{{\mathscr C}(\lambda_{{\mathcal P}} Y\backslash Y)}$.
Therefore $g=f^{-1}$. Since $g$  is an  order-homomorphism, it follows that $f$ is an  order-isomorphism.
\end{proof}

\begin{remark}\label{KHGD}
The list of topological properties ${\mathcal P}$ satisfying the assumption of Theorem \ref{FHFH} (as well as Theorem \ref{FJHFH} below) include
all topological properties (1)--(10) introduced in Example \ref{JHL} (see Remark \ref{LJG}).
\end{remark}

The following  generalizes Theorem 3.14 of  \cite{Ko2}.

\begin{theorem}\label{FJHFH}
Let $X$ and $Y$ be Tychonoff locally-${\mathcal P}$ spaces,  where  ${\mathcal P}$ is either pseudocompactness or a closed hereditary topological
property, which is preserved under  finite closed sums of subspaces. The following are equivalent:
\begin{itemize}
\item[\rm(1)] $\beta X\backslash\lambda_{{\mathcal P}} X$ and $\beta Y\backslash\lambda_{{\mathcal P}} Y$ are homeomorphic.
\item[\rm(2)] $({\mathscr E}_{{\mathcal P}-far}(X),\leq)$ and $({\mathscr E}_{{\mathcal P}-far}(Y),\leq)$ are order-isomorphic.
\end{itemize}
\end{theorem}

\begin{proof}
By Lemmas \ref{JIUH}, \ref{HGJHG} and \ref{JHG} (and since $X$ and $Y$ are locally-${\mathcal P}$) we have $X\subseteq\lambda_{{\mathcal P}} X$ and $Y\subseteq\lambda_{{\mathcal P}} Y$. Thus, by Lemma \ref{SGFH}, we have
\[\Theta_X\big({\mathscr E}_{{\mathcal P-far}}(X)\big)={\mathscr C}(\beta X\backslash\lambda_{{\mathcal P}} X)\backslash\{\emptyset\}\mbox{ and }\Theta_Y\big({\mathscr E}_{{\mathcal P-far}}(Y)\big)={\mathscr C}(\beta Y\backslash\lambda_{{\mathcal P}} Y)\backslash\{\emptyset\}.\]
The theorem now follows, as the order-structure of the set of closed subsets of any Hausdorff space determines its topology
(see Theorem 11.1 of \cite{B}).
\end{proof}

\section{The lattice of one-point ${\mathcal P}$-extensions}

In \cite{MRW1}, for a Tychonoff space $X$ and a topological property ${\mathcal P}$, J. Mack, M. Rayburn and R.G.  Woods have studied the lattice of one-point ${\mathcal P}$-extensions of $X$. The following theorem is from \cite{MRW1}. Recall that  a space $X$ is called
{\em ${\mathcal P}$-regular} (when ${\mathcal P}$ is a  topological property) if $X$ is a subspace of some ${\mathcal P}$-space.

\begin{theorem}[Mack, Rayburn and Woods \cite{MRW1}]\label{HGGKH}
Let $X$ be a Tychonoff locally-${\mathcal P}$ ${\mathcal P}$-regular space where ${\mathcal P}$ is a closed  hereditary Mr\'{o}wka topological property such that
\begin{itemize}
\item[\rm(1)] any product of compact ${\mathcal P}$-spaces is a  ${\mathcal P}$-space
\item[\rm(2)] each ${\mathcal P}$-space is a subspace of a compact ${\mathcal P}$-space
\item[\rm(3)] any product of a ${\mathcal P}$-space and a compact ${\mathcal P}$-space is a  ${\mathcal P}$-space
\item[\rm(4)] if $Y_1$ and $Y_2$ are closed $G_\delta$ subspaces of a compact ${\mathcal P}$-space $T$ and $Y$ is any subspace of $T$ for which $Y\cup Y_i$ is a ${\mathcal P}$-space for $i=1,2$ then $Y\cup(Y_1\cap Y_2)$ is a  ${\mathcal P}$-space.
\end{itemize}
Then $({\mathscr E}_{{\mathcal P}}(X),\leq)$ is a lattice which is also a complete upper semi-lattice. Moreover,
$({\mathscr E}_{{\mathcal P}}(X),\leq)$ is a complete lattice if and only if $X$ is locally compact.
\end{theorem}

Motivated by the above theorem, our next result deals with the lattice of one-point ${\mathcal P}$-extensions of a Tychonoff  space.

\begin{theorem}\label{GKH}
Let $X$ be a Tychonoff locally-${\mathcal P}$ non-${\mathcal P}$ space, where  ${\mathcal P}$ is  either pseudocompactness or a  closed  hereditary
Mr\'{o}wka topological  property, which is preserved under finite closed sums of subspaces. Then
$({\mathscr E}_{{\mathcal P}}(X),\leq)$ is a lattice, which is also a complete upper semi-lattice. Moreover,  $({\mathscr E}_{{\mathcal P}}(X),\leq)$
is a complete lattice if and only if $X$ is locally compact.
\end{theorem}

\begin{proof}
The fact that $({\mathscr E}_{{\mathcal P}}(X),\leq)$ is a complete upper semi-lattice follows from Lemma \ref{KJGT}. Also,
if $Y_i\in {\mathscr E}_{{\mathcal P}}(X)$, for $i=1,2$, then by Lemma \ref{JHKFH} we have
\[\Theta_X^{-1}\big(\Theta_X(Y_1)\cap\Theta_X(Y_2)\big)=Y_1\vee Y_2\mbox{ and }\Theta_X^{-1}\big(\Theta_X(Y_1)\cup\Theta_X(Y_2)\big)=Y_1\wedge Y_2.\]
Thus $({\mathscr E}_{{\mathcal P}}(X),\leq)$ is a lattice.  Note that a complete upper semi-lattice is a complete lattice if and
only if it has a minimum (see Theorem 2.1 (e) of \cite{PW}). Therefore $({\mathscr E}_{{\mathcal P}}(X),\leq)$ is a complete lattice if and only if
$({\mathscr E}_{{\mathcal P}}(X),\leq)$  has a minimum if and only if (by Lemma \ref{JHKFH}) there is a largest non-empty compact subset of
$\beta X\backslash X$ containing $\beta X\backslash\lambda_{{\mathcal P}} X$ if and only if $\beta X\backslash X$ is compact if and only if $X$ is
locally compact.
\end{proof}

\begin{remark}\label{HKHGD}
Topological properties (1)--(10) introduced in Example \ref{JHL} all satisfy the assumption of Theorem \ref{GKH} (see Remark \ref{LJG}).
\end{remark}

\section{Conjecture}

We conclude this article with the following conjecture.

\begin{conjecture}\label{HGFGKH}
Let $X$ and $Y$ be locally compact spaces. The following are equivalent:
\begin{itemize}
\item[\rm(1)] $\mbox{\em cl}_{\beta X}(\beta X\backslash\upsilon X)$ and $\mbox{\em cl}_{\beta Y}(\beta Y\backslash\upsilon Y)$ are homeomorphic.
\item[\rm(2)] $({\mathscr E}^*(X),\leq)$ and $({\mathscr E}^*(Y),\leq)$ are order-isomorphic.
\end{itemize}
\end{conjecture}

\begin{remark}\label{POGOHD}
Conjecture \ref{HGFGKH} is different from the form it was originally stated in the early version of the article, due to the referee's comments. In the original formulation of the conjecture the spaces $X$ and $Y$ were only assumed to be Tychonoff. The referee has given examples showing that the conjecture, as it was originally stated, is indeed wrong. The rest of this section (along with the present reformulation of the conjecture) is due to the referee, who, very kindly, allowed (in fact suggested) the author to include them here as part of the article. The author is greatly indebted to the referee for this generosity, the elegant construction of the examples, and his comments and suggestions which led to considerable improvements in this section.
\end{remark}

\begin{remark}\label{UHD}
Note that if $X$ is a Tychonoff space  then $({\mathscr E}^*(X),\leq)$ and $({\mathscr E}^*(\upsilon X),\leq)$ are order-isomorphic. This is because the zero-sets of $\beta X$ which are contained in $\beta X\backslash X$ are subsets of $\beta X\backslash\upsilon X$, and these zero-sets determine the order structure of either $({\mathscr E}^*(X),\leq)$ or $({\mathscr E}^*(\upsilon X),\leq)$. Thus, as far as $({\mathscr E}^*(X),\leq)$ is concerned, one can assume that the space $X$ is realcompact.
\end{remark}

In the following, we give examples of Tychonoff spaces $X$ and $Y$ for which condition (1) in Conjecture \ref{HGFGKH} holds though condition (2) fails. Note that by Remark \ref{UHD} we can assume that spaces under consideration are all realcompact.

\begin{example}\label{YGFSW}
There are two countable spaces $X$ and $Y$, one of which is locally compact (in fact, discrete), such that $\mbox{cl}_{\beta X}(\beta X\backslash\upsilon X)$ is homeomorphic to $\mbox{cl}_{\beta Y}(\beta Y\backslash\upsilon Y)$ but $({\mathscr E}^*(X),\leq)$ is not order-isomorphic to $({\mathscr E}^*(Y),\leq)$.

For this example, note that if $X$ is (realcompact but) non-compact, the partially ordered set $({\mathscr E}^*(X),\leq)$ has a smallest element if and only if $X$ is locally compact and $\sigma$-compact, that is, if and only if $\beta X\backslash X$ is a zero-set of $\beta X$. To see this, note first that if $\beta X\backslash X$ is a zero-set of $\beta X$, and if $T\in{\mathscr E}^*(X)$, then $\Theta_X(T)$ is a zero-set of $\beta X$ which is contained in, but not equal to, $\beta X\backslash X$. Therefore, there exists a zero-set $Z$ of $\beta X$ which is disjoint from $\Theta_X(T)$. The element $\Theta_X^{-1}(Z)$ of ${\mathscr E}^*(X)$ is not comparable to $T$.

For the example, let $X=\mathbf{N}$ and $Y=\mathbf{N}\cup\{u\}$, where $u\in\beta\mathbf{N}\backslash\mathbf{N}$. Both $X$ and $Y$ are countable, and therefore, realcompact and $X$ is locally compact, while $Y$ is not. By what was just observed, $({\mathscr E}^*(X),\leq)$ has a smallest element, but $({\mathscr E}^*(Y),\leq)$ does not. So, the  partially ordered sets $({\mathscr E}^*(X),\leq)$ and $({\mathscr E}^*(Y),\leq)$ are not order-isomorphic. However,
\[\mbox{cl}_{\beta X}(\beta X\backslash\upsilon X)=\beta\mathbf{N}\backslash\mathbf{N}=\mbox{cl}_{\beta Y}(\beta Y\backslash\upsilon Y).\]
\end{example}

The above example might leave the impression that the  partially ordered sets $({\mathscr E}^*(X),\leq)$ and $({\mathscr E}^*(Y),\leq)$ can fail to be order-isomorphic only because of the difference between {\em local compactness} and {\em non-local compactness}. In the next example, neither space is locally compact at any point.  (Recall that a Tychonoff space $X$ is said to be {\em locally compact at a point $x\in X$} if $x$ has an open neighborhood $U$ in $X$ with a compact closure.) For a Tychonoff space $X$, if we denote by $R(X)$ the set of all  points of $X$ at which $X$ is locally compact, then it is well known that for any compactification $\alpha X$ of $X$ we have
\[R(X)=X\cap\mbox{cl}_{\alpha X}(\alpha X\backslash X).\]

\begin{example}\label{IHF}
There are two Lindel\"{o}f nowhere locally compact spaces $X$ and $Y$, one of which is countable such that $\mbox{cl}_{\beta X}(\beta X\backslash\upsilon X)$ is homeomorphic to $\mbox{cl}_{\beta Y}(\beta Y\backslash\upsilon Y)$ but $({\mathscr E}^*(X),\leq)$ is not order-isomorphic to $({\mathscr E}^*(Y),\leq)$.

Recall that in a partially ordered set $(P,\leq)$ a set $D\subseteq P$ is said to be {\em dense} if for each $p\in P$ there exists a $d\in D$ such that $d\leq p$. Let us say that a partially ordered set $(P,\leq)$ {\em has property $(\star)$} if there exists a countable subset $C$ of $P$ such that for each $p\in P$ there exists a $c\in C$ and a $z\in P$ such that $z\geq p$ and $z\geq c$. In other words, $(P,\leq)$ has property ($\star$) if there exists a countable subset $C$ of $P$ such that the set
\[\{z\in P:z\geq c\mbox{ for some }c\in C\}\]
is dense in  $P$ partially ordered with the reverse order of $\leq$.

Let $X=\mathbf{Q}\cap \mathbf{I}$. Of course, $X$ is homeomorphic to $\mathbf{Q}$, as it is a dense in itself countable metrizable space (see Problem 6.2.A of \cite{E}). Let $\phi:\beta X\rightarrow\mathbf{I}$ be the Stone extension of the identity map. Let $D$ be a countable dense subset of $\mathbf{I}\backslash X$ and let $Y=\phi^{-1}(\mathbf{I}\backslash D)$. Then of course $X$ is countable, and, $Y$, being the inverse image under a perfect mapping of a Lindel\"{o}f space, is Lindel\"{o}f.  Neither spaces is locally compact at any point. This is evident for $X$; we prove it for $Y$ by showing that
\begin{equation}\label{UGWE}
\mbox{cl}_{\beta Y}(\beta Y\backslash Y)=\beta Y.
\end{equation}
First, note that since $X\subseteq\phi^{-1}(X)\subseteq\phi^{-1}(\mathbf{I}\backslash D)=Y$ and $Y\subseteq\beta X$ we have $\beta Y=\beta X$. To show (\ref{UGWE}), note that
\[\beta Y\backslash Y=\beta X\backslash \phi^{-1}(\mathbf{I}\backslash D)=\phi^{-1}(D)\]
and $\phi^{-1}(D)$ is dense in $\beta X$, as it is the inverse image of the dense subset $D$ of $\mathbf{I}$ under $\phi$, and the mapping $\phi$, being the Stone extension of the identity map, is irreducible (see Lemma 6.5 (b) of \cite{PW}). Next, we show that $({\mathscr E}^*(Y),\leq)$ has property ($\star$) while $({\mathscr E}^*(X),\leq)$ does not. Note that $\phi^{-1}(d)$ is non-empty, for $d\in D$, as $\phi$ is onto, as its image is compact and thus closed in $\mathbf{I}$ and it contains the dense subset $X$ of  $\mathbf{I}$; also $\phi^{-1}(d)\in{\mathscr Z}(\beta X)$, as $\{d\}\in{\mathscr Z}(\mathbf{I})$ and $\phi^{-1}(d)\subseteq\beta Y\backslash Y$, as  $\phi^{-1}(d)\subseteq\phi^{-1}(D)$ and by above $Y\cap\phi^{-1}(D)=\emptyset$.
Let
\[{\mathscr C}=\big\{\Theta_Y^{-1}\big(\phi^{-1}(d)\big):d\in D\big\}.\]
Then ${\mathscr C}\subseteq{\mathscr E}^*(Y)$ and ${\mathscr C}$ is countable. Let $T\in{\mathscr E}^*(Y)$. Note that
\[\bigcup_{d\in D}\phi^{-1}(d)\cup Y=\phi^{-1}(D)\cup \phi^{-1}(\mathbf{I}\backslash D)=\phi^{-1}(\mathbf{I})=\beta Y.\]
Therefore $\{\phi^{-1}(d):d\in D\}$ partitions $\beta Y\backslash Y$. There exists  some $d\in D$ such that $\phi^{-1}(d)\cap \Theta_Y^{-1}(T)\neq\emptyset$. Let $C=\Theta_Y^{-1}(\phi^{-1}(d))$. Then $C\in{\mathscr C}$. Note that $\Theta_Y(T)\cap\phi^{-1}(d)$ is a non-empty zero-set of $\beta Y$ which misses $Y$. Thus $\Theta_Y(T)\cap\phi^{-1}(d)=\Theta_Y(Z)$ for some $Z\in{\mathscr E}^*(Y)$. Now since  $\Theta_Y(Z)\subseteq\Theta_Y(T)$ and $\Theta_Y(Z)\subseteq\phi^{-1}(d)=\Theta_Y(C)$ it follows that $Z\geq T$ and $Z\geq C$. This shows that $({\mathscr E}^*(Y),\leq)$ has property ($\star$). Next, to show that $({\mathscr E}^*(X),\leq)$ fails to have property ($\star$), suppose to the contrary that there exists some countable ${\mathscr C}'\subseteq{\mathscr E}^*(X)$ with the property that for each $T'\in{\mathscr E}^*(X)$ there exist some $C'\in {\mathscr C}'$ and some $Z'\in {\mathscr E}^*(X)$ such that $Z'\geq T'$ and $Z'\geq C'$. Note that
\[\bigcup_{C'\in{\mathscr C}'}\Theta_X(C')\subsetneq \beta X\backslash X\]
as $X$ is not \v{C}ech-complete (see Problem 3.9.B of \cite{E}). Choose some
\[p\in (\beta X\backslash X)\backslash\bigcup_{C'\in{\mathscr C}'}\Theta_X(C').\]
Since $X$ is realcompact, there exists some $A\in{\mathscr Z}(\beta X)$ which contains $p$ and misses $X$.  Let  $B\in{\mathscr Z}(\beta X)$ be such that $p\in B$ and  $B\cap\Theta_X(C')=\emptyset$ for every $C'\in{\mathscr C}'$. Now $A\cap B$ is a non-empty zero-set of $\beta X$ which misses $X$ and thus $A\cap B=\Theta_X(T')$ for some $T'\in{\mathscr E}^*(X)$. Now, if there exist some $C'\in {\mathscr C}'$ and some $Z'\in {\mathscr E}^*(X)$ such that $Z'\geq T'$ and $Z'\geq C'$, then $\Theta_X(Z)\subseteq\Theta_X(C')$ and $\Theta_X(Z)\subseteq A\cap B$. But this is not possible, as $A\cap B\cap\Theta_X(C')=\emptyset$. This  shows that $({\mathscr E}^*(X),\leq)$ does not have property ($\star$) and completes the proof.
\end{example}

\begin{remark}\label{JHD}
If Conjecture \ref{HGFGKH} were correct in ZFC, the result would be very important, for the following reason. It is not known whether it is consistent that  $\beta\mathbf{N}\backslash\mathbf{N}$ is homeomorphic to  $\beta D(\aleph_1)\backslash D(\aleph_1)$, where $D(\aleph_1)$ is the discrete space of cardinality $\aleph_1$. In {\em Open Problems in Topology II} \cite{Ny},  P.J. Nyikos has proposed what he calls {\em Axiom $\Omega$}, which is the assertion that these spaces are homeomorphic. (Again, it is not known whether or not Axiom $\Omega$ is consistent.) If the conjecture is correct in ZFC, then Axiom $\Omega$ fails in ZFC. The reason is that by the comments at the beginning of Example \ref{YGFSW}, the partially ordered set $({\mathscr E}^*(\mathbf{N}),\leq)$ has a smallest element, whereas $({\mathscr E}^*(D(\aleph_1)),\leq)$ does not.
\end{remark}

\noindent{\bf Acknowledgements}\\

The author is greatly indebted to the anonymous referee for careful reading of the paper and the prompt report (given within two weeks). It is a pleasure for the author to repeat his thanks to the referee for his/her comments and suggestions (particularly, the extensive comments regarding Conjecture \ref{HGFGKH}) and for  his/her  permission (in fact, suggestion) to include Examples \ref{YGFSW} and \ref{IHF} in this paper.


\begin{thebibliography}{10}

\bibitem{Be} G. Beer, On convergence to infinity. Monatsh. Math. 129 (2000) 267--280.

\bibitem{B} G. Birkhoff, Lattice Theory. American Mathematical Society, New York, 1948.

\bibitem{BS} R.L. Blair and M.A. Swardson,  Spaces with an Oz Stone-\v{C}ech  compactification. Topology Appl. 36 (1990) 73--92.

\bibitem{Bu} D.K. Burke, Covering properties, in: K. Kunen and J.E. Vaughan (Eds.), Handbook of Set-theoretic Topology,  Elsevier,  Amsterdam, 1984, pp. 347--422.

\bibitem{C} W.W. Comfort, On the Hewitt realcompactification of a product space. Trans. Amer. Math. Soc. 131 (1968) 107--118.

\bibitem{CN} W.W. Comfort and S. Negrepontis, The Theory of Ultrafilters. Springer-Verlag, New York, 1974.

\bibitem{E} R. Engelking, General Topology. Second edition. Heldermann Verlag, Berlin, 1989.

\bibitem{F} Z. Frol\'{\i}k, A generalization of realcompact spaces. Czechoslovak Math. J. 13 (1963) 127--138.

\bibitem{GJ} L. Gillman and M. Jerison, Rings of Continuous Functions.  Springer-Verlag, New York-Heidelberg, 1976.

\bibitem{HJ} A.W. Hager and D.G.  Johnson, A note on certain subalgebras of $C(X)$. Canad. J. Math. 20 (1968) 389--393.

\bibitem{Ha} D. Harris, The local compactness of $\upsilon X$. Pacific J. Math. 50 (1974) 469--476.

\bibitem{HJW} M. Henriksen, L. Janos and R.G. Woods, Properties of one-point completions of a non-compact metrizable space. Comment. Math. Univ. Carolin. 46 (2005) 105--123.

\bibitem{Ko1} M.R. Koushesh, On one-point metrizable extensions of locally compact metrizable spaces. Topology Appl. 154 (2007) 698--721.

\bibitem{Ko2} M.R. Koushesh, On order-structure of the set of one-point Tychonoff extensions of a locally compact space. Topology Appl. 154 (2007) 2607--2634.

\bibitem{Ko3} M.R. Koushesh, Compactification-like extensions (submitted).

\bibitem{Ko4} M.R. Koushesh, One-point extensions of locally compact paracompact spaces (accepted).

\bibitem{MRW1} J. Mack, M. Rayburn and R.G.  Woods, Local topological properties and one-point extensions. Canad. J. Math. 24 (1972) 338--348.

\bibitem{MRW2} J. Mack, M. Rayburn and R.G.  Woods,  Lattices of topological extensions. Trans. Amer. Math. Soc. 189 (1974) 163--174.

\bibitem{M} K. Morita, Countably-compactifiable spaces. Sci. Rep. Tokyo Kyoiku Daigaku Sect. A 12 (1973) 7--15.

\bibitem{M1}  S. Mr\'{o}wka, Some comments on the author's example of a non-$R$-compact space. Bull. Acad. Polon. Sci. S\'{e}r. Sci. Math. Astronom. Phys. 18 (1970) 443--448.

\bibitem{Mr} S. Mr\'{o}wka, On local topological properties. Bull. Acad. Polon. Sci. 5 (1957) 951--956.

\bibitem{Ny} P.J. Nyikos, \v{C}ech-Stone remainders of discrete spaces, in: E. Pearl (Ed.), Open Problems in Topology II,  Elsevier,  Amsterdam, 2007, pp. 207--216.

\bibitem{PW} J.R. Porter and R.G. Woods, Extensions and Absolutes of Hausdorff Spaces. Springer-Verlag, New York, 1988.

\bibitem{Steph} R.M. Stephenson, Jr., Initially $\kappa$-compact and related spaces, in: K. Kunen and J.E. Vaughan (Eds.), Handbook of Set-theoretic Topology, Elsevier, Amsterdam, 1984, pp. 603--632.

\bibitem{Va} J.E. Vaughan, Countably compact and sequentially compact spaces, in: K. Kunen and J.E. Vaughan (Eds.), Handbook of Set-theoretic Topology,  Elsevier,  Amsterdam, 1984, pp. 569--602.

\bibitem{W} M.D. Weir,  Hewitt-Nachbin Spaces. American Elsevier, New York, 1975.

\end{thebibliography}
\end{document}